\documentclass[a4paper,12pt]{article}

\usepackage{amsmath,amssymb,amscd,amsthm,makeidx,txfonts,graphicx,url,bm}
\usepackage{a4wide,pstricks,xy}
\usepackage{mathtools,tikz-cd}

\usepackage{array}   
\newcolumntype{L}{>{$}l<{$}} 
\newcolumntype{C}{>{$}c<{$}} 

\usepackage[backref]{hyperref}
\numberwithin{equation}{subsection}

\xyoption{all}
\input{xypic}

\newtheorem{theorem}{Theorem}[section]
\newtheorem{proposition}[theorem]{Proposition}
\newtheorem{corollary}[theorem]{Corollary}

\newtheorem{lemma}[theorem]{Lemma}

\theoremstyle{remark}
\newtheorem{remark}[theorem]{Remark}

\theoremstyle{definition}

\newtheorem{nothing}[subsubsection]{}

\def\beq{\begin{eqnarray}}
\def\eeq{\end{eqnarray}}
\def\bes{\begin{eqnarray*}}
\def\ees{\end{eqnarray*}}

\newcommand{\Tau}{\mathcal{T}}

\def\omhat{{\bm\omega}}

\def\muhat{{\bm \mu}}

\def\lambdahat{{\bm \lambda}}

\def\calV{{\mathcal{V}}}

\def\t{{\rm type}}
\def\bfG{{\bf G}}
\def\bfT{{\bf T}}

\DeclareMathOperator{\ch}{ch}

\def\w{{\bf w}}

\def\C{\mathbb{C}}

\def\calQ{{\mathcal{Q}}}

\def\calX{{\mathcal{X}}}
\def\calU{{\mathcal{U}}}

\def\calV{{\mathcal{V}}}

\def\calF{{\mathcal{F}}}

\def\calP{\mathcal{P}}

\def\x{\mathbf{x}}

\def\w{\mathbf{w}}

\def\H{\mathbb{H}}

\def\N{\mathbb{Z}_{\geq 0}}

\def\F{\mathbb{F}}
\def\Q{\mathbb{Q}}
\def\T{\mathbb{T}}
\def\Tr{{\rm Tr}}

\def\t{\mathfrak{t}}
\def\calC{{\mathcal C}}

\def\calO{{\mathcal O}}

\def\Z{\mathbb{Z}}

\def\K{\mathbb{K}}

\def\gl{{\mathfrak g\mathfrak l}}

\def\g{{\mathfrak{g}}}

\newcommand{\nc}{\newcommand}

\def\GU{{\rm GU}}

\def\bS{{\bf S}}

\def\bM{{\mathbb{M}}}
\def\bH{{\mathbb{H}}}

\nc{\op}[1]{\mathop{\mathchoice{\mbox{\rm #1}}{\mbox{\rm #1}}
{\mbox{\rm \scriptsize #1}}{\mbox{\rm \tiny #1}}}\nolimits}
\nc{\al}{\alpha}

\nc{\ep}{\varepsilon} \nc{\ga}{\gamma} \nc{\Ga}{\Gamma}
\nc{\la}{\lambda} \nc{\La}{\Lambda} \nc{\si}{\sigma}
\nc{\Sig}{{\Gamma}} \nc{\Om}{\Omega} \nc{\om}{\omega}
\nc{\Frob}{\op{ Frob}}

\nc{\SL}{{\rm SL}} \nc{\GL}{{\rm GL}} \nc{\PGL}{{\rm PGL}}
\nc{\G}{{\rm G}}

\def\x{{\bf x}}

\nc{\cpt}{{\op{cpt}}} \nc{\Dol}{{\op{Dol}}} \nc{\DR}{{\op{DR}}}
\nc{\B}{{\op{B}}} \nc{\Triv}{\op{Triv}} \nc{\Hod}{{\op{Hod}}}
\nc{\Log}{{\op{Log}}} \nc{\Exp}{{\op{Exp}}} \nc{\Est}{E_{\op{st}}}
\nc{\Hst}{H_{\op{st}}} \nc{\Left}[1]{\hbox{$\left#1\vbox to
  10.5pt{}\right.\nulldelimiterspace=0pt \mathsurround=0pt$}}
\nc{\Right}[1]{\hbox{$\left.\vbox to
  10.5pt{}\right#1\nulldelimiterspace=0pt \mathsurround=0pt$}}
\nc{\LEFT}[1]{\hbox{$\left#1\vbox to15.5pt{}\right.\nulldelimiterspace=0pt \mathsurround=0pt$}}
\nc{\RIGHT}[1]{\hbox{$\left.\vbox to
 15.5pt{}\right#1\nulldelimiterspace=0pt \mathsurround=0pt$}}

\nc{\bee}{{\bf E}} \nc{\bphi}{{\bf \Phi}}

\begin{document}

\title{Ennola duality for  decomposition of  tensor products}

\author{Emmanuel Letellier \\ {\it
  Universit\'e Paris Cit\'e, IMJ-PRG} \\ {\tt
  emmanuel.letellier@imj-prg.fr}\and Fernando Rodriguez-Villegas
  \\ 
{\it  ICTP Trieste} \\ {\tt villegas@ictp.it}  }

 \pagestyle{myheadings}

\maketitle

\begin{abstract} 
  Ennola duality relates the character table of the finite unitary
  group $\GU_n(\F_q)$ to that of $\GL_n(\F_q)$ where we replace $q$ by
  $-q$ (see \cite{E} for the original observation and \cite{LS} for
  its proof). The aim of this paper is to investigate Ennola duality
  for the decomposition of tensor products of irreducible
  characters. It does not hold just by replacing $q$ by $-q$. The main
  result of this paper is the construction of a family of two-variable
  polynomials $\Tau_\muhat(u,q)$ indexed by triples of partitions of
  $n$ which interpolates between multiplicities in decompositions of tensor
  products of unipotent characters for $\GL_n(\F_q)$ and
  $\GU_n(\F_q)$.  We give a module theoretic interpretation of these
  polynomials and deduce that they have non-negative integer
  coefficients. We also deduce that the coefficient of the term of
  highest degree in $u$ equals the corresponding Kronecker coefficient
  for the symmetric group and that the constant term in $u$ give
  multiplicities in tensor products of generic irreducible characters
  of unipotent type (i.e. unipotent characters twisted by linear
  characters of $\GL_1(\F_q)$).
 
\end{abstract}

\newpage
\tableofcontents
\newpage

\section{Introduction}

Let $G=\GL_n(\overline{\F}_q)$ and consider the
two geometric Frobenius endomorphisms
$$
F:G\rightarrow G, \qquad (g_{ij})\mapsto (g_{ij}^q)
$$
 and
$$
F':G\rightarrow G, \qquad  g\mapsto F({^t}g^{-1}).
$$ 
One of the main goals of this work is to study decomposition of tensor products of irreducible characters for
the finite groups
$$
\GL_n(\F_q)=G^F, \qquad \qquad \GU_n(\F_q)=G^{F'}
$$
and compare them.

Ennola duality states that one can obtain the character table of
$\GU(\F_q)$ from that of $\GL_n(\F_q)$ by essentially replacing
$q$ by $-q$. (Ennola's conjecture was proved by Lusztig and Srinivasan
in~\cite{LS}). A natural question is then:
\begin{center}
{\it To what extent does Ennola duality extend to the character rings
  of $\GL_n(\F_q)$ and $\GU_n(\F_q)$?} 
\end{center}
Examples show that simply replacing $q$ by $-q$ does not preserve the
multiplicities of the tensor product of characters of $\GL_n(\F_q)$ and
their counterparts of $\GU_n(\F_q)$. For example, for $n=4$, thanks to the
tables in~\cite{mattig}, we see that
$$
\left\langle {\rm St}\otimes{\rm St},{\rm
    St}\right\rangle_{G^F}=q^3+2q+1,\hspace{1cm}\left\langle {\rm
    St}\otimes{\rm St},{\rm St}\right\rangle_{G^{F'}}=q^3+1 
$$
where ${\rm St}$ denotes the Steinberg character. Therefore, if there
is some extension of Ennola duality to the character rings it must be
more involved.

Since 

$$
\left\langle \calX_1\otimes\calX_2,\calX_3\right\rangle=\left\langle\calX_1\otimes\calX_2\otimes\calX_3^*,1\right\rangle
$$
where $\calX_3^*$ is the dual character, we will study multiplicities of the form $\left\langle\calX_1\otimes\cdots\otimes\calX_k,1\right\rangle$ for a $k$-tuple of irreducible characters of either $\GL_n(\F_q)$ or $\GU_n(\F_q)$.
\bigskip

Our first result is that for {\it generic} $k$-tuples of irreducible characters the
situation is straightfoward: the multiplicities for the tensor product
of an arbitrary number of such characters are given by certain
polynomials $V_\omhat(q)$ and $V'_\omhat(q)$ respectively, which satisfy
$$
V'_\omhat(q)=\pm V_\omhat(-q)
$$
with an explicit sign (see Corollary~\ref{Ennola-gen} for the precise
formulation). As we see in the above example, a formula of this sort
does not hold for arbitrary characters.

Our second result (see Theorem \ref{maintheo}) is that the polynomials
$V_\omhat(t)$ and $V'_\omhat(t)$ are obtained from a $q$-graded
$\C[\bS_n\times \langle\iota\rangle]$-module ${\mathbb M}^\bullet_n$,
where~$\iota$ is an involution and
$\bS_n:=(S_n)^k=S_n\times \cdots \times S_n$. Namely,
$$
\mathbb{M}^j_n:=H_c^{2j+d_n}(\calQ_n,\C)\otimes(\varepsilon^{\boxtimes k}),
$$
where $\varepsilon$ is the sign representation of $S_n$ and where $\calQ_n$ is a certain {\it generic} non-singular
irreducible affine algebraic (quiver) variety of dimension $d_n$.
\bigskip

In order to state our third main result we need to set some
notation.  For a partition $\mu$ of $n$ let $\calU_\mu,\calU'_\mu$ be
the corresponding unipotent character of $G^F$ and $G^{F'}$
respectively (the Steinberg character corresponds to the partition
$(1^n)$).  In \cite{L2} Letellier proved that, for any multi-partition
$\muhat=(\mu^1,\dots,\mu^k)$ of $n$ the multiplicity

\begin{equation}
\label{U-defn}
U_\muhat(q):=\left\langle\calU_{\mu^1}\otimes
\cdots\otimes\calU_{\mu^k},1\right\rangle_{G^F},
\end{equation}
can be computed in terms of the master series $\Omega$ of~\cite{HLRV2}
and~\cite{HLRV}  as follows

\begin{equation}
\label{GL-eqn}
1+\sum_{n>0}\sum_\muhat U_\muhat(q)\, s_\muhat
T^n=\Exp\left(\Psi\right),
\qquad
\Psi:=(q-1)\Log\left(\Omega\right) =\sum_{n>0}\sum_\muhat V_\muhat(q)\, s_\muhat T^n,
\end{equation}
where $\muhat$ runs through $k$-tuples of partitions of $n$. Here $V_\mu(q)$
are the multiplicities (as in~\eqref{U-defn}) for {\it generic}
unipotent characters (i.e., twisted by appropriate $1$-dimensional
characters) and $s_\muhat$ denote the multi-Schur function
$s_\muhat:=s_{\mu^1}(\x^1)\cdots s_{\mu^k}(\x^k)$ in the ring of
symmetric function $\Lambda=\Lambda(\x^1,\ldots,\x^k)$ in the $k$ sets
of infinitely many variables $\x^1,\ldots, \x^k$ (see \S \ref{comb}).

To obtain the corresponding relation for $\GU_n(\F_q)$, we introduce an
extra variable $u$ and define $\Tau_n(x;u,q)\in \Lambda[u,q]$ by
\begin{equation}
\label{GL-u}
\Exp\left(u\Psi \right)=1+u\sum_{n\geq 1} \Tau_n(x;u,q)\, T^n.
\end{equation}
For convenience we also define
$$
\Tau_\muhat(u,q):=\langle\Tau_n(u,q),s_\muhat\rangle
$$
for a multipartion $\muhat$. We prove that $\Tau_\muhat(u,q)$ are  polynomials in the variables $u$ and $q$ with non-negative integer coefficients (see Formula (\ref{tau-for})).

In this setup the identity~\eqref{GL-eqn} is the following statement (see Theorem \ref{K-coeff}(i))
\begin{equation}
\label{V-Tau}
V_\muhat(q)=\Tau_\muhat(0,q),\qquad \qquad U_\muhat(q)=\Tau_\muhat(1,q).
\end{equation}

Let now
$$
U'_\muhat(q):=\left\langle\calU'_{\mu^1}\otimes
\cdots\otimes\calU'_{\mu^k},1\right\rangle_{G^{F'}}
$$
be the multiplicities for unipotent characters of the unitary group
$\GU_n(\F_q)$. 

Our third result is the following, which we can consider
as the version of Ennola duality for the character rings of $\GL_n$
and $\GU_n$ over finite fields. We have (see Theorem \ref{EnnolaII})
$$
U'_\muhat(q)=\pm\Tau_\muhat(-1,-q)
$$
(with an explicit formula for the sign).
\bigskip

As an illustration of our results, here is a list of a few values of
$$
\tau_n:=\langle\Tau_n,s_{1^n}(\x^1)s_{1^n}(\x^2)\rangle
$$
with~$k=3$ (so a symmetric function in one remaining set of infinitely
many variables).  We give these in two different formats for better
readability:
$$
\begin{array}{c|l}
n& \tau_n\\
\hline
2& us_{2}+s_{1^2} \\
3& u^2s_{3}+ (u + 1)s_{2\cdot1}+ (u + q)s_{1^3} \\
4& u^3s_{4}+ (u^2 + u + 1)s_{3\cdot1}+(2u + q)s_{2^2}+ 
(q^2 + uq + q + u^2 + u + 1)s_{2\cdot1^2}\\
&\quad +(uq +   u + q^3 + q)s_{1^4} 
\end{array}
$$
$$
\begin{array}{c|l}
n& \tau_n\\
\hline
2& us_{2}+s_{1^2} \\
3& u^2s_{3}+ u(s_{2\cdot1} + s_{1^3})+  qs_{1^3} + s_{2\cdot1}\\
4& u^3s_{4}+ u^2(s_{3\cdot1}+s_{2\cdot12}) + uq(s_{2\cdot1^2}+s_{1^4})
+u(s_{3\cdot1}+2s_{2^2} +   s_{2\cdot1^2} +s_{1^4})\\
&\quad + q^3s_{1^4}
+q^2s_{2\cdot1^2}
+q(s_{2^2} +   s_{2\cdot1^2} +s_{1^4})
+s_{3\cdot1} +   s_{2\cdot1^2} 
\end{array}
$$
For example, we have $\langle\tau_4,s_{1^4}(\x^3)\rangle=uq +   u + q^3 +
q$. Evaluating this polynomial at $u=0,1,-1$ we find
$$
u=0\qquad q^3+q; \qquad \qquad u=1\qquad q^3+2q+1;\qquad \quad u=-1,\qquad q^3 -1,
$$
matching the respective values of 
$$
V_{1^4,1^4,1^4}(q),\qquad U_{1^4,1^4,1^4}(q), \qquad -U'_{1^4,1^4,1^4}(-q),
$$
 in the tables in~\S\ref{examples}.

 In our fourth and final result, we show (see Theorem~\ref{K-coeff})
 that the coefficient of $u^{n-1}$ in $\Tau_\muhat(u,q)$ (the largest
 possible power of $u$, so basically evaluating $\Tau_\muhat$ at
 $u=\infty$) is independent of $q$ and equals the Kronecker
 coefficient
$$
\langle\chi^{\mu^1}\otimes\cdots\otimes\chi^{\mu^k},1\rangle_{S_n},
$$
where $\muhat=(\mu^1,\dots,\mu^k)$ and where $\chi^{\mu^i}$ denotes the irreducible character of $S_n$ corresponding to the partition $\mu^i$ (the character $\chi^{(1^n)}$ being the sign character).
\bigskip

In this paper we are interested in the two extreme cases, namely the multiplicities for generic characters and the multiplicities for unipotent characters (the least generic). There are also the intermediate cases studied by T. Scognamiglio \cite{Sco} who introduced the technical notion of level of genericity (at least for the split characters). The introduction of the variable $u$ make also sense for these intermediate cases. This is not more complicated as what we do here for unipotent characters but these intermediate cases are much more technical to define. 
\bigskip

A natural question (suggested by the results of this paper) is whether
the polynomials $\Tau_\muhat(u,q)$ are structure coefficients of some
based ring, which
would be a $u$-deformation of the character ring of $\GL_n(\F_q)$.
More precisely, by the previous work of Hausel-Letellier-Villegas
\cite{HLRV1}, Letellier \cite{L1}\cite{L2} and Scognamiglio \cite{Sco}
we may reconstruct (the structure coefficients of) the character ring
of the groups $\GL_n(\F_q)$ (where $n$ runs over $\mathbb{N}^*$) from
the generic structure coefficients of a given type (for instance the
semisimple regular, or the semisimple split, or the unipotent
type). With this we can define $u$-deformations of all the structure
coefficients of the character ring of the $\GL_n(\F_q)$'s and speculate
whether they are the structure coefficients of some based ring.

As a further speculation, we may ask whether this $u$-deformation is
related to that between Bosons ($u=1$) and Fermions ($u=-1$) in
physics~\cite{Za92}.


\bigskip

\paragraph{Acknowledgements:} A part of this work was done while the
first author was visiting the Sydney Mathematical Research
Institute. The first author is grateful to the SMRI for the wonderful
research environment and their generous support.  The second author
would like to thank the Universit\'e Paris Cit\'e, where this work was
started, for its hospitality.

\section{Preliminaries}

Let $G$ denotes $\GL_n(\overline{\F}_q)$ and consider the two Frobenius endomorphisms

$$
F:G\longrightarrow G,\hspace{.5cm} (a_{ij})\mapsto (a_{ij}^q),\hspace{1cm}F':G\longrightarrow G,\hspace{.5cm}(a_{ij})\mapsto{^t}(a_{ij}^q)^{-1}.
$$
We let $\ell$ be a prime which does not divide $q$. We will consider representations of finite groups over $\overline{\Q}_\ell$-vector spaces and for a finite group $H$ we denote by $\widehat{H}$ its set of irreducible characters.  For a field $k$ we denote by $k^*$ the group of non-zero elements.

\subsection{Combinatorics}\label{comb}
\bigskip

{\bf Partitions, types, symmetric functions}
\bigskip

We denote by $\calP$ the set of all partitions of integers including the unique partition $0$, by $\calP_n$ the set of partitions of $n$. Partitions $\lambda$ are denoted by $\lambda=(\lambda_1,\lambda_2,\dots)$, where $\lambda_1\geq\lambda_2\geq\cdots\geq 0$. We will also write a partition $\lambda$ as $(1^{m_1},2^{m_2},\dots)$ where $m_i$ denotes the multiplicity of $i$ in $\lambda$. 
\bigskip

For a partition $\lambda$ of $n$, we denote by $\chi^\lambda$ the corresponding  irreducible character of $S_n$ (the partition $(n^1)$ corresponds to the trivial character and the partition $(1^n)$ corresponds to the sign character).

\bigskip

We will denote by $|\lambda|=\sum_i\lambda_i$ the size of $\lambda$ and by $\lambda^*$ the dual partition of $\lambda$. We will put

$$
n(\lambda):=\sum_{i>0}(i-1)\lambda_i.
$$
A \emph{type} is a function 

$$
\omega:\Z_{>0}\times(\calP\backslash\{0\})\longrightarrow \N
$$
with finite support. We will write $\omega$  as

$$
\omega=\{(d_i,\omega^i)^{m_i}\}_i
$$
where $m_i$ denotes the image of $(d_i,\omega^i)\in\Z_{>0}\times(\calP\backslash\{0\})$.

The size of $\omega$ is defined as 

$$
|\omega|:=\sum_im_id_i|\omega^i|.
$$
and we denote by $\omega^*$ the dual type $\{(d_i,\omega^{i*})^{m_ i}\}_ i$.

We denote by $\T_n$ the set of types of size $n$ and for a type $\omega=\{(d_i,\omega^i)^{m_i}\}_i$ we introduce 

\begin{equation}
n(\omega):=\sum_im_id_i n(\omega^i),\hspace{1cm}r(\omega):=n+\sum_im_i|\omega^i|,\hspace{1cm}r'(\omega):=\lceil n/2\rceil+\sum_im_i|\omega^i|.
\label{r}\end{equation}
For an infinite set of commuting variables $\x=\{x_1,x_2,\dots\}$, we denote by $\Lambda(\x)$ the ring of symmetric functions in the variables of $\x$. It is equiped with the Hall pairing $\langle\,,\,\rangle$ that makes the Schur symmetric functions $\{s_\lambda(\x)\}$ an orthonormal basis.
\bigskip

The transformed Hall-Littlewood symmetric function $\tilde{H}_\lambda(\x;q)\in\Lambda(\x)\otimes_\Z\Q(q)$ is defined as

$$
\tilde{H}_\lambda(\x;q):=\sum_\nu\tilde{K}_{\nu\lambda}(q)s_\nu(\x)
$$
where $\tilde{K}_{\nu\lambda}(q)=q^{n(\lambda)}K_{\nu\lambda}(q^{-1})$ are the transformed Kostka polynomials \cite[III (7.11)]{macdonald}.

\bigskip

Given a family of symmetric functions $u_\lambda(\x;q)\in\Lambda(\x)\otimes_\Z\Q(q)$ indexed by partitions $\lambda$ (with $u_0=1$), we extend it to a type $\omega=\{(d_i,\omega^i)^{m_i}\}$ by

$$
u_\omega(\x;q)=\prod_iu_{\omega^i}(\x^{d_i};q^{d_i})^{m_i}
$$
where $\x^d$ denotes the set of variables $\{x_1^d,x_2^d,\dots\}$.
\bigskip

Consider now $k$ separate sets $\x_1,\x_2,\dots,\x_k$ of infinitely many variables and denote by 

$$
\Lambda=\Q(q)\otimes_\Z\Lambda(\x_1)\otimes_\Z\cdots\otimes_\Z\Lambda(\x_k)
$$
the ring of functions separately symmetric in each set $\x_1,\x_2,\dots,\x_k$ with coefficients in $\Q(q)$. Denote by $\langle\,,\,\rangle_i$ the Hall pairing on $\Lambda(\x_i)$ and consider
the Hall pairing
$$
\langle\,,\,\rangle:=\prod_i\langle\,,\,\rangle_i
$$
on $\Lambda$.
\bigskip

Given a family of functions $u_\lambda(\x_1,\dots,\x_k,q)\in\Lambda$ indexed by partitions with $u_0=1$. We extend its definition to a type $\tau=\{(d_i,\tau^i)^{m_i}\}_{i=1,\dots,r}\in\T_n$ by 

$$
u_\tau(\x_1,\dots,\x_k,q):=\prod_{i=1}^ru_{\tau^i}(\x_1^{d_i},\dots,\x_k^{d_i},q^{d_i}).
$$

{\bf Exp and Log}
\bigskip

\noindent Consider $$\psi_n:\Lambda[[T]]\rightarrow\Lambda[[T]],\, f(\x_1,\dots,\x_k;q,T)\mapsto f(\x_1^n,\dots,\x_k^n;q^n,T^n).$$ The $\psi_n$ are called the \emph{Adams operations}.

Define $\Psi:T\Lambda[[T]]\rightarrow T\Lambda[[T]]$ by $$\Psi(f)=\sum_{n\geq 1}\frac{\psi_n(f)}{n}.$$Its inverse is given by $$\Psi^{-1}(f)=\sum_{n\geq 1}\mu(n)\frac{\psi_n(f)}{n}$$where $\mu$ is the ordinary M\"obius function. 

Define $\Log:1+T\Lambda[[T]]\rightarrow T\Lambda[[T]]$ and its inverse $\Exp:T\Lambda[[T]]\rightarrow 1+\Lambda[[T]]$ as 

$$\Log(f)=\Psi^{-1}\left(\log(f)\right)$$and $$\Exp(f)=\exp\left(\Psi(f)\right).$$

\begin{remark}The map $T\mapsto -T$ is not preserved under $\Log$ and $\Exp$ as 

$$
1+q^iT^j=(1-q^{2i}T^{2j})/(1-q^iT^j).
$$
\label{-}\end{remark}

For a type $\tau=\{(d_i,\tau^i)^{m_i}\}\in\T_n$, we put

\begin{equation}
c_\tau^o:=\begin{cases}\frac{(-1)^{r-1}\mu(d)(r-1)!}{d\prod_i m_i!}&\text{ if for all } i, d_i=d,\\
0&\text{ otherwise.}\end{cases}
\label{ctau}\end{equation}

By \cite[Formula (2.3.9)]{HLRV} we have the following.

\begin{proposition}Assume given a family of functions  $u_\lambda=u_\lambda(\x_1,\dots,\x_k;q)\in\Lambda$ is indexed by partitions with $u_0=1$. Then
\begin{equation}\Log\left(\sum_{\lambda\in \calP}u_\lambda T^{|\lambda|}\right)=\sum_\tau c_\tau^o u_\tau T^{|\tau|}
\label{Log}\end{equation}
where $\tau$ runs over the set of types of size larger or equal to $1$.
\end{proposition}
\bigskip

We also recall the following result of Mozgovoy \cite[Lemma 22]{Mozgovoy}.
\bigskip

For $h\in \Lambda$ and $n\geq 1$ we put 

$$
h_n:=\frac{1}{n}\sum_{d | n}\mu(d)\psi_{\frac{n}{d}}(h).
$$
This is the M\"obius inversion formula of 

$$
\psi_n(h)=\sum_{d| n}d\cdot h_d.
$$

\begin{lemma} Let $h\in \Lambda$ and $f_1,f_2\in 1+T\Lambda[[T]]$ such that

$$
\log\,(f_1)=\sum_{d=1}^\infty h_d\cdot\log\,(\psi_d(f_2)).
$$
Then 
$$
\Log\,(f_1)=h\cdot \Log\,(f_2).
$$
\label{moz}\end{lemma}

{\bf Cauchy function}

\bigskip

\noindent  The $k$-points Cauchy function is
defined as
\begin{equation}\Omega(q)=\Omega(\x_1,\dots,\x_k,q;T):=\sum_{\lambda\in\calP}\frac{1}{a_\lambda(q)}\left(\prod_{i=1}^k\tilde{H}_{\lambda}(\x_i,q)\right)T^{|\lambda|}\in1+T\Lambda[[T]]  
\label{Cauchyfunction}\end{equation} 
where $a_\lambda(q)$ is a polynomial in $q$ which gives the cardinality of the centralizer of a
unipotent element of $\GL_n(\F_q)$ with Jordan form of type $\lambda$
\cite[IV, (2.7)]{macdonald}.

For a family of symmetric functions $u_\lambda(\x;q)$ indexed by partitions and  a multi-type $\omhat=(\omega_1,\dots,\omega_k)\in\big(\T_n\big)^k$, we put 
$$
u_{\omhat}:=u_{\omhat_1}(\x_1,q)\cdots u_{\omhat_k}(\x_k,q)\in\Lambda.
$$

For $\omhat=(\omega_1,\dots,\omega_k)\in\big(\T_n\big)^k$, with $\omega_i=\{(d_{ij},\omega_i^j)^{m_{ij}}\}_{j=1,\dots,r_i}$, define

\begin{align*}
\H_\omhat(q):&=(q-1)\left\langle \Log\,\Omega(q),s_\omhat\right\rangle\\
&=(q-1)\sum_{\tau\in\T_n}c_\tau^o\frac{1}{a_\tau(q)}\left\langle\prod_{i=1}^k\tilde{H}_\tau(\x_i;q),s_\omhat\right\rangle.\end{align*}

where  $\left\langle \Log\,\Omega(q),s_\omhat\right\rangle$ is the Hall pairing of $s_\omhat$ with the coefficient of $\Log\,\Omega(q)$ in $T^n$.
 \bigskip
 
 The term $a_\tau(q)=\prod_ia_{\tau^i}(q^{d_i})$ is the cardinality of the centralizer in $\GL_{|\tau|}(\F_q)$ of an element of type $\tau$.

\subsection{The characters of $\GL_n(\F_q)=G^F$}\label{typeG}
\bigskip

{\bf Conjugacy classes}
\bigskip

 \noindent Let $\Xi$ denote the set of $F$-orbits of
$\overline{\F}_q^*=\GL_1(\overline{\F}_q)$ and for an integer $m\geq 0$, we denote by
$\calP_m(\Xi)$ the set of all maps $f:\Xi\rightarrow\calP$ such that

$$
|f|:=\sum_{\xi\in\Xi}|\xi|\, |f(\xi)|=m
$$
where $|\xi| $ denotes the size of the $F$-orbit $\xi$. The set $\calP_n(\Xi)$ parametrizes naturally  the set of conjugacy classes of $G^F$ using Jordan decomposition. For $f\in\calP_n(\Xi)$, we denote by $C_f$ the corresponding conjugacy class of $G^F$.
\bigskip

For instance, the conjugacy classes of 

$$
\left(\begin{array}{cccc}
x&1&0&0\\
0&x&0&0\\
0&0&x^q&1\\
0&0&0&x^q\end{array}\right)
$$
with $x\in\F_{q^2}\backslash\F_q$, corresponds to $\Xi\rightarrow\calP$ that maps the $F$-orbit $\{x,x^q\}$ to the partition $(2^1)$ and the other $F$-orbits to $0$.
\bigskip

For $f\in\calP_m(\Xi)$ and  a pair $(d,\lambda)\in \Z_{>0}\times(\calP\backslash\{0\})$, we put

$$
m_{d,\lambda}:=\#\{\theta\in\Theta\,|\, |\theta| =d, f(\theta)=\lambda\}
$$
The collection of the multiplicities $m_{d,\lambda}$ defines a type $\t(f)\in\T_m$ called the \emph{type} of $f$. 
\bigskip

For example, the elements of  $\T_2$ are $(1,1)^2$, $(2,1)$, $(1,1^2)$ and $(1,2^1)$ and are the types of the following kind of matrices (up to conjugacy in $\GL_2(\overline{\F}_q)$)
\bigskip

$$
\left(\begin{array}{cc}a&0\\0&b\end{array}\right),\,\,\left(\begin{array}{cc}x&0\\0&x^q\end{array}\right),\,\,\left(\begin{array}{cc}a&0\\0&a\end{array}\right),\,\, \left(\begin{array}{cc}a&1\\0&a\end{array}\right)
$$
where $a\neq b\in\F_q^*$, $x\in\F_{q^2}\backslash\F_q$.
\bigskip

{\bf Irreducible characters}
\bigskip

\noindent We now review the parametrization of the irreducible characters.
\bigskip

\noindent For each integer $r>0$ we denote by $\F_{q^r}$ the unique subfield of $\overline{\F}_q$ of cardinality $q^r$. 
\bigskip

For integers $r$ and $s$ such that $r\mid s$ we have the norm map

$$
N_{r,s}:(\F_{q^s})^*\longrightarrow(\F_{q^r})^*,\hspace{1cm}x\mapsto x^{q^s-1/q^r-1}
$$
which is surjective.

It induces an injective map $\widehat{\F_{q^r}^*}\rightarrow\widehat{\F_{q^s}^*}$ and we consider the direct limit

$$
\Gamma=\varinjlim \widehat{\F_{q^r}^*}
$$
of the $\widehat{\F_{q^r}^*}$ via these maps. The Frobenius automorphism $F$ acts on $\Gamma$ by $\alpha\mapsto\alpha^q$ and we denote by $\Theta$ the set of $F$-orbits of $\Gamma$.

For an integer $m\geq 0$, we denote by $\calP_m(\Theta)$ the set of all maps $f:\Theta\rightarrow\calP$ such that

$$
|f|:=\sum_{\theta\in\Theta}|\theta|\, |f(\theta)|=m
$$
where $|\theta|$ denotes the size of the $F$-orbit $\theta$. 
\bigskip

As for $\calP_m(\Xi)$, we define a type $\t(f)\in\T_m$ for any $f\in\calP_m(\Theta)$.
\bigskip

The irreducible complex characters of $G^F$ are naturally parametrized by the set $\calP_n(\Theta)$ as we now recall.
\bigskip

For $f\in\calP_n(\Theta)$, we recall (see \cite{LS}) the construction of the corresponding irreducible character $\calX_f$ using Deligne-Lusztig theory.

Consider

$$
L_f^F:=\prod_{\theta\in\Theta, f(\theta)\neq 0}\GL_{|f(\theta)|}(\F_{q^{|\theta|}})
$$
This is the group of $\F_q$-points of an $F$-stable  Levi subgroup $L_f$ of (some parabolic subgroup of) $\GL_n(\overline{\F}_q)$. Choose a representative $\dot{\theta}$ of each $\theta\in\Theta$ such that $f(\theta)\neq 0$. The collection of the $\dot{\theta}$ composed with the determinant defines a linear character $\theta_f$ of $L_f^F$ while the collection of partitions $f(\theta)$ define a unipotent character $\calU_f$ of $L_f^F$ as follows. 

We get the corresponding unipotent character $\calU_\mu$ of $\GL_m(\F_q)$ as

\begin{equation}
\calU_\mu=\frac{1}{|S_m|}\sum_{w\in S_m}\chi^\mu(w) R_{T_w^F}^{\GL_m^F}(1)
\label{almostbis}\end{equation}
where $T_w$ is an $F$-stable maximal torus of $\GL_m$ obtained by twisting the torus of diagonal matrices by $w$ and where $R_{T_w^F}^{\GL_m^F}(1)$ is the Deligne-Lusztig induced of the trivial character.

Then $\calU_f$ is the external tensor product of the $\calU_{f(\theta)}$ where $\theta$ runs over the set $\{\theta\in\Theta\,|\, f(\theta)\neq 0\}$.

By \cite{LS} we have the following.

\begin{equation}
\calX_f=(-1)^{r(f)}R_{L_f^F}^{G^F}(\theta_f\cdot\calU_f),
\label{LS1}\end{equation}
where $r(f):=r(\t(f))$ is given by Formula (\ref{r}) and where for any $F$-stable Levi subgroup $L$ of $G$, we denote by
$R_{L^F}^{G^F}$ the Lusztig induction studied for instance in
\cite{DM}. Notice that $\sum_\theta |f(\theta)|$ is the $\F_q$-rank of
$L_f$ and that the right hand side of (\ref{LS1}) does not depend on
the choice of the representatives $\dot{\theta}$.

We will say that $(L_f^F,\theta_f,\calU_f)$ is a triple defining $\calX_f$.
\bigskip

For any irreducible character $\calX=\calX_h$, with $h\in\calP_n(\Theta)$, we define the character

$$
\tilde{\calX}:=(-1)^{r(h)}R_{L_h^F}^{G^F}(\calU_h).
$$
It does not depend on $\theta_h$ (it depends only on the type of $h$), it is not irreducible in general and takes the same values as $\calX$ at unipotent elements.

\begin{theorem}Let $\calX$ be an irreducible character of type $\omega$.

(1) For any conjugacy class $C$ of type $\tau$, we have

$$
\tilde{\calX}(C)=(-1)^{r(\omega)}\left\langle \tilde{H}_\tau(\x;q),s_\omega(\x)\right\rangle.
$$

(2) In particular

$$
\calX(1)=\tilde{\calX}(1)=\frac{q^{n(\omega)}\prod_{i=1}^n(q^i-1)}{H_\omega(q)},
$$
where for a partition $\lambda$, $H_\lambda(q)=\prod_{s\in\lambda}(q^{h(s)}-1)$ is the hook polynomial \cite[Chapter I, Part 3, Example 2]{macdonald}.

\label{xtilde}\end{theorem}

If $d_i=1$ for all $i$, the first assertion of the Theorem is  \cite[Theorem 2.2.2]{HLRV1}, otherwise the same proof works with slight modifications. The second assertion is standard \cite[Chapter IV, (6.7)]{macdonald}.

\subsection{The characters of $\GU_n(\F_q)=G^{F'}$}
\bigskip

{\bf Conjugacy classes}
\bigskip

Denote by $\Xi'$ the set of $F'$-orbits of $\overline{\F}_q^*=\GL_1(\overline{\F}_q)$ and for $\xi\in\Xi'$, denote by $|\xi| $ the cardinal of $\xi$.  The set of conjugacy classes of $G^{F'}$ is in bijection with the set 

$$
\calP_n(\Xi'):=\left\{f:\Xi'\rightarrow\calP\,\left|\, \sum_{\xi\in\Xi'}|\xi| \, |f(\xi)|=n\right\}\right..
$$
For $f\in\calP_n(\Xi)$, we let $C'_f$ be the corresponding conjugacy class of $G^{F'}$. As in \S \ref{typeG} we can associate to any $f\in\calP_n(\Xi')$ a type $\t(f)\in\T_n$.

For example, the types $(1,1)^2$, $(2,1)$, $(1,1^2)$ and $(1,2^1)$ are respectively the types of the following kind of matrices (up to conjugacy in $\GL_2(\overline{\F}_q)$)
\bigskip

$$
\left(\begin{array}{cc}a&0\\0&b\end{array}\right),\,\,\left(\begin{array}{cc}x&0\\0&x^{-q}\end{array}\right),\,\,\left(\begin{array}{cc}a&0\\0&a\end{array}\right),\,\, \left(\begin{array}{cc}a&1\\0&a\end{array}\right)
$$
where $a\neq b\in\mu_{q+1}$, $x\in\F_{q^2}\backslash\mu_{q+1}$.
\bigskip

For a type $\tau$ of size $n$, we define the polynomial

\begin{equation}
a'_\tau(t):=(-1)^na_\tau(-t)
\label{a'}\end{equation}
By Wall (see \cite[Proposition 3.2]{NV}), the evaluation $a'_\tau(q)$ is the cardinality of the centralizer of an element of $G^{F'}$ of type $\tau$.
\bigskip

{\bf Irreducible characters}
\bigskip

\noindent Let us now give the construction of the irreducible characters of $G^{F'}$.

For a positive integer, we consider the multiplicative group

$$
M_m:=\{x\in\overline{\F}_q^*\,|\, x^{q^m}=x^{(-1)^m}\}.
$$
We have $M_m=\F_{q^m}^*$ if $m$ is even and $M_m=\mu_{q^m+1}$ if $m$ is odd.

If $r\mid m$, then the polynomial $|M_r|$ divides $|M_m|$ and  we have a norm map

$$
M_m\rightarrow M_r,\hspace{.5cm} x\mapsto x^{|M_m|/|M_r|}.
$$
We may thus consider the direct limit

$$
\Gamma':=\varinjlim \widehat{M_m}
$$
of the character groups $\widehat{M_m}$. The Frobenius $F':x\mapsto x^{-q}$ on $\overline{\F}_q^*$ preserves the subgroups $M_m$ and so acts on $\Gamma'$. We denote by $\Theta'$ the set of $F'$-orbits of $\Gamma'$.

We denote by $\calP_m(\Theta')$ the set of all maps $f:\Theta'\rightarrow \calP$ such that

$$
|f|:=\sum_{\theta\in\Theta'} |\theta| \, |f(\theta)|=m.
$$

\bigskip

As in \S \ref{typeG}, we can associate to any $f\in\calP_m(\Theta')$ a type $\t(f)\in\T_m$.
\bigskip

The irreducible characters of $G^{F'}$ are naturally parametrized by the set $\calP_n(\Theta')$ (the trivial unipotent character corresponds to the partition $(n^1)$).
\bigskip

For $f\in\calP_n(\Theta')$, we construct the associated irreducible character $\calX'_f$ in terms of Deligne-Lusztig theory as follows. Define

$$
L_f^{F'}:=\prod_{\begin{tiny}\begin{array}{c}\theta\in\Theta', f(\theta)\neq 0\\|\theta| \text{ even}\end{array}\end{tiny}}\GL_{|f(\theta)|}\left(\F_{q^{|\theta| }}\right)\prod_{\begin{tiny}\begin{array}{c}\theta\in\Theta', f(\theta)\neq 0\\|\theta| \text{ odd}\end{array}\end{tiny}}{\GU}_{|f(\theta)|}\left(\F_{q^{|\theta| }}\right)
$$
This is the group of $\F_q$-points of some $F'$-stable Levi subgroup $L_f$ of $G$. For each $\theta\in\Theta'$ such that $f(\theta)\neq 0$, choose a representative $\dot{\theta}$ of $\theta$. The collection of the $\dot{\theta}$ composed with the determinant defines a linear character $\theta'_f$ of $L_f^{F'}$ and the partitions $f(\theta)$ defines an almost unipotent character $\calU''_f$ of $L_f^{F'}$ using Formula (\ref{almostbis}) for both $F$ and $F'$.
\bigskip

For example, assume that $n=2$. If $\t(f)=(1,1)^2$, then $f$ is supported on two orbits of $\Theta'$ of size one, say $\{\alpha\}$ and $\{\beta\}$ with $\alpha,\beta\in\widehat{\mu_{q+1}}$, $L_f^{F'}\simeq \mu_{q+1}\times\mu_{q+1}$ and $\theta_f(a,b)= \alpha(a)\beta(b)$. If $\omega_f=(2,1)$, then $f$ is supported on one orbit $\{\eta,\eta^{-q}\}\in\Theta'$ of size $2$ with $\eta\in\widehat{\F}_{q^2}^*$, $L_f^{F'}\simeq \GL_1(\F_{q^2})$, and $\theta'_f=\alpha$.

\begin{remark}From \cite{LS}, the virtual character $\calU''_f$ is up to a sign a true unipotent character  of $L_f^{F'}$ which we denote by $\calU'_f$. For a partition $\mu$ of $n$ we have

$$
\calU'_\mu=(-1)^{n(\mu^*)}\calU_\mu''.
$$
The values of $\calU''_f$ at unipotent elements are obtained from those of $\calU_f$ essentially by replacing $q$ by $-q$. 
\end{remark}

\begin{theorem}\cite[Lusztig-Srinivasan]{LS} We have

$$
\calX'_f=(-1)^{r'(f)+n(f^*)}R_{L_f^{F'}}^{G^{F'}}(\theta'_f\cdot\calU''_f)
$$
where $r'(f):=r'(\t(f))$ is given by Formula (\ref{r}), $f^*\in\calP_n(\Theta')$ is obtained from $f$ by requiring that $f^*(\theta)$ is the dual partition $f(\theta)^*$ for each $\theta$, and where for any $f$, we put $n(f)=n(\t(f))$.
\label{LS}\end{theorem}

In \cite{LS}, it is proved that  $R_{L_f^{F'}}^{G^{F'}}(\theta'_f\cdot\calU''_f)$ is an irreducible true character up to a sign. The explicit computation of the sign in the above theorem is done in  \cite[Theorem 4.3]{NV}.
\bigskip

For an irreducible character $\calX'=\calX'_f$ of $G^{F'}$, define

$$
\tilde{\calX}'=(-1)^{r'(f)+n(f^*)}R_{L_f^{F'}}^{G^{F'}}(\calU''_f).
$$
We have the following theorem analogous to Theorem \ref{xtilde} with the Frobenius $F'$ instead of $F$.

\begin{theorem}[Ennola duality]Let $\calX'$ and $\calX$ be irreducible characters respectively of $G^{F'}$ and $G^F$ both of type $\omega$.

(1) For any conjugacy class $C'$ of $G^{F'}$ and $C$ of $G^F$ of type $\tau$, we have

\begin{align*}
\tilde{\calX}'(C')&=(-1)^{n(\omega^*)+{n\choose 2}}\,\tilde{\calX}(C)(-q)\\
&=(-1)^{r'(\omega)+n(\omega^*)}\left\langle \tilde{H}_\tau(\x;-q),s_\omega(\x)\right\rangle\\
\end{align*}

(2) In particular

$$
\calX'(1)=(-1)^{n(\omega^*)+{n\choose 2}}\,\calX(1)(-q)
$$
\label{charval'}\end{theorem}

\begin{proof}Using the character formula for Deligne-Lusztig induction \cite[\S 10.1]{DM}, the proof of the theorem reduces to Ennola duality for unipotent characters.

\end{proof}

\begin{remark}Note that as we know from Ennola duality that
  $\calX'(1)$ and $\calX(1)(-q)$ differ by a sign and that $\calX'(1)$
  is positive we can easily compute the sign  in (1) and in Theorem
  \ref{LS} from Theorem~\ref{xtilde}(2).

\end{remark}

\subsection{Ennola duality for generic multiplicities}

In \cite[Definition 4.2.2]{HLRV} we define the notion of generic $k$-tuple of irreducible characters of $G^F$. We define generic $k$-tuple of irreducible characters of $G^{F'}$ excatly in the same way. We do not give the definition as  we will only use the theorem below. However to give a taste of what it is we give the definition for irreductible characters whose type is a partition of $n$ (i.e. unipotent characters tensorized by a linear character of $G^F$).
\bigskip

If $\calU_1,\dots,\calU_k$ are unipotent characters of $G^F$ and if $\alpha\in\widehat{\GL_1(\F_q)}$ is of order $n$, then $$(\calU_1,\dots,\calU_{k-1},(\alpha\circ\det)\cdot\calU_k)$$ is a generic $k$-tuple of irreducible characters of $G^F$. 

\begin{remark}For any multi-type $\omhat=(\omega_1,\dots,\omega_k)\in(\T_n)^k$, there always exist generic $k$-tuples of irreducible characters of $G^F$ (or $G^{F'}$) of type $\omhat$ as long as the characteristic is large enough. The existence is equivalent to that of the existence of generic $k$-tuple of conjugacy classes (see \cite[Proposition 8.1.2]{LetS}) and the condition on the characteristic for the existence of generic $k$-tuple of conjugacy classes is explained in \cite[see above Proposition 3.4]{Lchar}. 
\label{small}\end{remark}
\bigskip

We have the following technical result:
\bigskip

\begin{theorem}(1) Let $(\calX_1,\dots,\calX_k)$ be a generic $k$-tuple of irreducible characters of $G^F$ of type $\omhat=(\omega_1,\dots,\omega_k)$. Let $\tau\in\T_n$ and denote by $C_\tau$ a conjugacy class of $G^F$ of type $\tau$. Then

\begin{align*}
\sum_{f\in\calP_n(\Xi), \t(f)=\tau}\prod_{i=1}^k\calX_i(C_f)&=(q-1) c^o_\tau\prod_{i=1}^k \tilde{\calX}_i(C_\tau)\\
&=(q-1)c_\tau^o(-1)^{r(\omhat)}\prod_{i=1}^k \left\langle \tilde{H}_\tau(\x_i;q),s_{\omega_i}\right\rangle.
\end{align*}
where $r(\omhat)=\sum_ir(\omega_i)$.

(2) Let $(\calX'_1,\dots,\calX'_k)$ be a generic $k$-tuple
  of irreducible characters of $G^{F'}$ of type
  $\omhat=(\omega_1,\dots,\omega_k)$. Let $\tau\in\T_n$ and denote by
  $C'_\tau$ a conjugacy class of $G^{F'}$ of type $\tau$. Then
\begin{align*}
\sum_{f\in\calP_n(\Xi'), \t(f)=\tau}\prod_{i=1}^k\calX'_i(C'_f)&=(q+1) c^o_\tau\prod_{i=1}^k \tilde{\calX}'_i(C'_\tau)\\
&=(q+1)c_\tau^o(-1)^{r'(\omhat)+n(\omhat^*)}\prod_{i=1}^k \left\langle \tilde{H}_\tau(\x_i;-q),s_{\omega_i}\right\rangle.
\end{align*}
where $r'(\omhat):=\sum_{i=1}^kr'(\omega_i)$ and $n(\omhat^*):=\sum_{i=1}^kn(\omega_i^*)$.
\label{sum-theo}\end{theorem}

\begin{proof}The assertion (1) follows from \cite[Lemma 2.3.5, Theorem 4.3.1]{HLRV} and the proof of the assertion (2) is completely similar.
\end{proof}

\begin{theorem}(1) Let $(\calX_1,\dots,\calX_k)$ be a generic $k$-tuple of irreducible characters of $G^F$ of type $\omhat\in(\T_n)^k$.  We have

$$
V_\omhat(q):=\left\langle\calX_1\otimes\cdots\otimes\calX_k,1\right\rangle_{G^F}=(-1)^{r(\omhat)}\H_\omhat(q).
$$
(2) Let $(\calX_1',\dots,\calX_k')$ be a generic $k$-tuple of irreducible characters of $G^{F'}$ of type $\omhat\in(\T_n)^k$.  We have

$$
V'_\omhat(q):=\left\langle\calX_1'\otimes\cdots\otimes\calX_k',1\right\rangle_{G^{F'}}=(-1)^{r'(\omhat)+n(\omhat^*)+n+1}\H_\omhat(-q).
$$
\label{genmulti}\end{theorem}
\bigskip

\begin{remark}According to Remark \ref{small}, generic $k$-tuples of irreducible characters of type $\omhat$ may not exist in small characteristics, however the  polynomials on the right hand side  of the above equalities  always exist and will be denoted by $V_\omhat(q)$ and $V'_\omhat(q)$ in small characteristics.

\end{remark}

The theorem says in particular that the generic multiplicities depend only on the types and not on the choice of the irreducible characters of a given type. Note that $\H_\omhat(q)$ is clearly a rational function in $q$ with rational coefficients. By the above theorem, it is also an integer for infinitely many values of $q$. Hence $\H_\omhat(q)$ is a polynomial in $q$ with rational coefficients. We will see that it has integer coefficients.

\bigskip

\begin{corollary}(Ennola duality for generic multiplicities)

$$
V'_\omhat(q)=(-1)^{r'(\omhat)+r(\omhat)+n(\omhat^*)+n+1}V_\omhat(-q).
$$
\label{Ennola-gen}\end{corollary}

In particular if $\omhat$ is a multi-partition $\muhat=(\mu^1,\dots,\mu^k)$, i.e. each coordinate $\omega_i$ is of the form $(1,\mu^i)$, then

$$
V'_\muhat(q)=(-1)^{k(n+\lceil n/2\rceil)+n(\muhat^*)+n+1}V_\muhat(-q).
$$

\begin{proof}[Proof of Theorem \ref{genmulti}]Assertion (1) was stated  without proof in \cite{L1}. We prove the assertion (2) for the convenience of the reader.

 We have

$$
\left\langle\calX'_1\otimes\cdots\otimes\calX'_k,1\right\rangle_{G^{F'}}=\sum_{C'}\frac{|C'|}{|G^{F'}|}\prod_{i=1}^k\calX'_i(C')
$$
where the sum is over the set over conjugacy classes. The quantity $|C'|/|G^{F'}|$ depends only on the type of $C'$, more precisely, see Formula (\ref{a'})
 $$
\frac{|C'_f|}{|G^{F'}|}=a'_{\t(f)}(q)^{-1}.
$$
We thus have

$$
\left\langle\calX'_1\otimes\cdots\otimes\calX'_k,1\right\rangle_{G^{F'}}=\sum_{\tau\in\T_n}\frac{1}{a'_\tau(q)}\sum_{f\in\calP_n(\Xi'), \t(f)=\tau}\prod_{i=1}^k\calX'_i(C'_f).
$$
Using Theorem \ref{sum-theo}(2) we get that

\begin{align*}
\left\langle\calX'_1\otimes\cdots\otimes\calX'_k,1\right\rangle_{G^{F'}}&=(-1)^{r'(\omhat)+n(\omhat^*)+n}(q+1)\left\langle\sum_{\tau\in\T_n}c_\tau^o\frac{1}{a_\tau(-q)}\prod_{i=1}^k\tilde{H}_\tau(\x_i;-q),s_\omhat\right\rangle\\
&=(-1)^{r'(\omhat)+n(\omhat^*)+n+1}\H_\omhat(-q).
\end{align*}

\end{proof}

\begin{remark}Notice that the map $q\mapsto -q$ is not preserved under $\Log$ (see Remark \ref{-}) and so we do not get $\H_\omhat(-q)$ as $(-q-1)\langle\Log(\Omega(-q)),s_\omhat\rangle$. 

\end{remark}

\section{Ennola duality for tensor products of unipotent characters}

\subsection{Infinite product formulas}\label{unipotent-GL}
\bigskip

For a multi-partition $\muhat=(\mu_1,\dots,\mu_k)$ of $n$, we consider the polynomials in $q$

$$
U_\muhat(q):=\left\langle \calU_{\mu^1}\otimes\cdots\otimes\calU_{\mu^k},1\right\rangle_{G^F},\hspace{1cm}U'_\muhat(q):=\left\langle \calU'_{\mu^1}\otimes\cdots\otimes\calU'_{\mu^k},1\right\rangle_{G^{F'}}.
$$
Let $\Phi_d(q)$, resp. $\Phi'_d(q)$, be the number of $F$-orbits, resp. $F'$-orbits, of $\overline{\F}_q^*$ of size $d\geq 1$.
\begin{proposition}
(1) We have \begin{equation}
\label{inf-prod-GL}
1+\sum_{n>0}\sum_{\muhat\in(\calP_n)^k}U_\muhat(q)s_\muhat
T^n=\prod_{d\geq 1}\Omega(\x_1^d,\dots,\x_k^d,q^d; T^d)^{\Phi_d(q)} 
\end{equation}
where $\Omega(\x_1,\dots,\x_k,q;T)$ is given by Formula (\ref{Cauchyfunction}).

\noindent (2) We have \begin{equation}
\label{inf-prod-U}
  1+\sum_{n>0}\sum_{\muhat\in(\calP_n)^k}
  (-1)^{\frac{1}{2}d_\muhat+1+n}U'_\muhat(q)s_\muhat\, T^n=\prod_{d\geq1}\Omega(\x_1^d,\dots,\x_k^d,(-q)^d;T^d)^{\Phi'_d(q)}
\end{equation}
where $$d_\muhat:=n^2(k-2)-\sum_{i,j}(\mu^i_j)^2+2.$$
\label{infinite}\end{proposition}

\begin{proof} Formula (\ref{inf-prod-GL})  is proved in \cite[Proof of Proposition 25]{L2}. Let us prove the second formula.

By
  Theorem \ref{charval'}(1), for a partition $\mu$ of size $n$ and
  conjugacy class $C'$ of $G^{F'}$ we have
$$
\calU'_\mu(C')=(-1)^{n+\lceil n/2\rceil+n(\mu^*)}\left\langle\tilde{H}_{\t(C')}(\x;-q),s_\mu(\x)\right\rangle
$$
Therefore by Equation ( \ref{a'}) we have
$$
1+\sum_{n>0}\sum_{\muhat\in(\calP_n)^k}(-1)^{\frac{1}{2}d_\muhat+1+n}U'_\muhat(q)s_\muhat\, T^n=\sum_{f\in\calP(\Xi')}\frac{1}{a_{\t(f)}(-q)}\prod_{i=1}^k\tilde{H}_{\t(f)}(\x_i;-q)T^{|\t(f)|}\\
$$
as 
$$
\frac{1}{2}d_\muhat+1\equiv k(n+\lceil n/2\rceil)+n(\muhat^*)\bmod 2.
$$

If $\omega=\{(d_i,\omega^i)^{m_i}\}$ is a type then
$$
a_\omega(q)=\prod_i a_{\omega^i}(q^{d_i})^{m_i}
$$
but $b_\omega(q):=a_\omega(-q)$ does not satisfy such an identity. Indeed $b_{\omega^i}(q^{d_i})=a_{\omega^i}(-q^{d_i})$ for both odd and even $d_i$ while
$$
b_\omega(q)=\prod_{i,\, d_i \text{ even}} a_{\omega^i}(q^{d_i})^{m_i}\prod_{i,\, d_i \text{ odd}}a_{\omega^i}(-q^{d_i})^{m_i}.
$$
Therefore we consider the partition
$$
\Xi'=\Xi'_e\coprod \Xi'_o
$$
into orbits of even and odd size respectively. Then
$$
\calP(\Xi')=\calP(\Xi'_e)\times\calP(\Xi'_o)
$$
and
\begin{align*}
1+\sum_{n>0}\sum_{\muhat\in(\calP_n)^k}&(-1)^{\frac{1}{2}d_\muhat+1+n}U'_\muhat(q)s_\muhat\, T^n\\
&=\left(\sum_{f\in\calP(\Xi'_e)}\frac{1}{a_{\t(f)}(q)}\prod_{i=1}^k\tilde{H}_{\t(f)}(\x_i;q)T^{|\t(f)|}\right)\left(\sum_{f\in\calP(\Xi'_o)}\frac{1}{a_{\t(f)}(-q)}\prod_{i=1}^k\tilde{H}_{\t(f)}(\x_i;-q)T^{|\t(f)|}\right)\\
&=\prod_{\xi\in\Xi'_e}\Omega\left(\x_1^{|\xi|},\dots,\x_k^{|\xi|},q^{|\xi|};T^{|\xi|}\right)\prod_{\xi\in\Xi'_o}\Omega\left(\x_1^{|\xi|},\dots,\x_k^{|\xi|},-q^{|\xi|};T^{|\xi|}\right)
\end{align*}
hence the result.
\end{proof}

\subsection{Ennola duality for tensor products of unipotent characters}

By M\"obius inversion formula we have

$$
\Phi_d(q)=\frac{1}{d}\sum_{r\mid d}\mu(r)(q^{d/r}-1),\hspace{1cm}\Phi'_d(q)=\frac{1}{d}\sum_{r\mid d}\mu(r)\left(q^{d/r}-(-1)^{d/r}\right).
$$
We introduce a new variable $u$ and we define a commun $u$-deformation of $\Phi_d(q)$ and $\Phi'_d(q)$ as

$$
\Phi_d(u,q):=\frac{1}{d}\sum_{r\mid d}\mu(r) u^{d/r}(q^{d/r}-1).
$$
Indeed

$$
\Phi_d(1,q)=\Phi_d(q),\hspace{1cm}\Phi_d(-1,-q)=\Phi'_d(q).
$$
For a multi-partition $\muhat$, define polynomials $\Tau_\muhat(u,q)$ by the formula 

\begin{equation}
\prod_{d\geq 1}\Omega(\x_1^d,\dots,\x_k^d,q^d;T^d)^{\Phi_d(u,q)}=1+u\sum_{n>0}\sum_{\muhat\in(\calP_n)^k}\Tau_\muhat(u,q) s_\muhat T^n.
\label{tau}\end{equation}
From Proposition \ref{infinite} we have the following.
\begin{theorem} [Ennola duality] We have
$$
U_\muhat(q)=\Tau_\muhat(1,q),\qquad \qquad
U_\muhat'(q)=(-1)^{\frac{1}{2}d_\muhat+n}\Tau_\muhat(-1,-q).
$$
\label{EnnolaII}\end{theorem}
We will also prove in  \S \ref{PK} the following result.

\begin{theorem} (i)   We have
$$
V_\muhat(q)=\Tau_\muhat(0,q),\qquad\qquad
V'_\muhat(q)=(-1)^{\frac{1}{2}d_\muhat+n}\,\Tau_\muhat(0,-q). 
$$
(ii) For a multi-partition $\muhat=(\mu^1,\dots,\mu^k)$, the
coefficient of the term of $\Tau_\muhat(u,q)$ of degree $n-1$ in $u$
is independent of $q$ and equals the Kronecker coefficient  
$$
\langle\chi^{\mu^1}\otimes\cdots\otimes\chi^{\mu^k},1\rangle_{S_n}.
$$
\label{K-coeff}\end{theorem}

\section{Module theoretic interpretation of the generic multiplicities}

\subsection{Quiver varieties}

\label{quiver}

Let $\K$ be an algebraically closed field ($\mathbb{C}$ or $\overline{\F}_q$). Fix a \emph{generic} $k$-tuple $(\calC_1,\dots,\calC_k)$ of semisimple regular adjoint orbits of $\gl_n(\K)$, i.e. the adjoint orbits $\calC_1,\dots,\calC_k$ are semisimple regular, 

$$
\sum_{i=1}^k\Tr(\calC_i)=0,
$$
and for any subspace $V$ of $\K^n$ stable by some $X_i\in\calC_i$ for each $i$ we have

$$
\sum_{i=1}^k\Tr(X_i|_V)\neq 0
$$
unless $V=0$ or $V=\K^n$ (see \cite[Lemma 2.2.2]{HLRV}). In other
words, the sum of the eigenvalues of the orbits
$\calC_1,\dots,\calC_k$ equals $0$ and if we select $r$ eigenvalues of
$\calC_i$ for each $i$ with $1\leq r< n$, then the sum of the selected
eigenvalues does not vanish. Such a $k$-tuple
$(\calC_1,\dots,\calC_k)$ always exists. 

Consider the affine algebraic variety
$$
\calV_n:=\left\{(X_1,\dots,X_k)\in\calC_1\times\cdots\times\calC_k\,\left|\, \sum_iX_i=0\right\}\right..
$$
The diagonal action of  $\GL_n(\K)$ on $\calV_n$ by conjugation induces a free action of $\PGL_n(\K)$ (in particular all $\GL_n$-orbits of $\calV$ are closed), see \cite[\S 2.2]{HLRV},  and we consider the GIT quotient

$$
\calQ=\calQ_n:=\calV_n/\!/\PGL_n(\K)={\rm Spec}\left(\K[\calV_n]^{\PGL_n(\K)}\right).
$$
This is a non-singular irreducible affine algebraic variety (see \cite[Theorem 2.2.4]{HLRV}) of dimension

\begin{equation}
{\rm dim}\, \calQ=n^2(k-2)-kn+2.
\label{dim}\end{equation}
Crawley-Boevey \cite{CB} makes a connection between the points of $\calQ$ and
representations of the star-shaped quiver with $k$-legs of length $n$
 from which the variety  $\calQ$ can be realized as a
quiver variety (see \cite{HLRV} and references therein for details). 

\bigskip

Denote by $H_c^*(\calQ)$ the compactly supported cohomology of $\calQ$ (if $\K=\C$, this is the usual cohomology with coefficients in $\C$ and if the characteristic of $\K$ is positive this is the $\ell$-adic cohomology with coefficients in $\overline{\Q}_\ell$). The variety $\calQ$ is cohomologically pure and has vanishing odd cohomology (see \cite[Section 2.4]{CBvdB} and \cite[Theorem 2.2.6]{HLRV}). 

\subsection{Action of $\bS'_n$ on cohomology}\label{Weyl-action}

Let $\K$ be either $\C$ or $\overline{\F}_q$. Consider the involutions $\GL_n(\K)\rightarrow \GL_n(\K)$, $g\mapsto {^t}g^{-1}$ and $\gl_n(\K)\rightarrow\gl_n(\K)$, $x\mapsto -{^t}x$ which we both denote by $\iota$. 
Notice that 
$$
\iota(gxg^{-1})=\iota(g)\iota(x)\iota(g)^{-1}
$$
for any $g\in\GL_n(\K)$ and $x\in\gl_n(\K)$. 

Notice also that $\iota$ fixes permutation matrices of $\GL_n(\K)$
which are identified with $S_n$. Consider the finite group
$$
\bS_n':=\bS_n\times\langle \iota\rangle.
$$
where $\bS_n:=(S_n)^k$.
\bigskip

In this section we construct an  action of $\bS'_n$ on the cohomology $H_c^*(\calQ)$ (notice that $\bS_n$ and $\langle\iota\rangle$ do not act on $\calQ$).
\bigskip

The construction of the $\bS_n$-action is done in  \cite{HLRV2} (this is a particular case of the action of Weyl groups on the cohomology of quiver varieties as studied by many authors including Nakajima \cite{Nakajima}\cite{Nakajima1}, Lusztig \cite{Lusztig} and Maffei \cite{Maffei}). The $\bS_n$-module structure does not depend on the choice of the eigenvalues of the orbits $\calC_1,\dots,\calC_k$ (as long as this choice is generic). 
\bigskip

Let $\t_n\subset\gl_n$ be the closed subvariety of diagonal matrices and let $\bm{\t}^{\rm gen}_n$ be the open subset of $\t_n^k$ of generic regular $k$-tuples $(\sigma_1,\dots,\sigma_k)$, i.e. for each $i=1,\dots,k$, the diagonal matrix $t_i$ has distinct eigenvalues and if $\calO_i$ denotes the $\GL_n$-orbit of $t_i$, then the $k$-tuple $(\calO_1,\dots,\calO_k)$ is generic.
\bigskip

Let $T_n\subset\GL_n$ be the closed subvariety of diagonal matrices and put

$$
\bfG_n=(\GL_n)^k,\hspace{1cm}\bfT_n=(T_n)^k,\hspace{1cm}\bm{\g}_n=(\gl_n)^k.
$$
Consider the GIT quotient

$$
\tilde{\calQ}_n:=\left.\left.\left\{\left(X,g\bfT_n,\sigma\right)\in\bm{\g}_n\times(\bfG_n/\bfT_n)\times\bm{\t}^{\rm gen}_n\,\left|\, g^{-1}Xg=\sigma, \sum_i X_i=0\right\}\right.\right/\!\right/\bfG_n
$$
where $\bfG_n$ acts by conjugation on $\bm{\g}_n$ and by left multiplication on $\bfG_n/\bfT_n$. 
\bigskip

The group $\bS_n$ acts on $\bfG_n/\bfT_n$ as $s\cdot g\bfT_n:=gs^{-1}\bfT_n$ where we regard elements of $S_n$ as permutation matrices in $\GL_n$. It acts also on $\bm{\t}^{\rm gen}_n$ by conjugation from which we get an action of $\bS_n$ on $\tilde{\calQ}_n$. 

The projection

$$
p:\tilde{\calQ}_n\rightarrow\bm{\t}^{\rm gen}_n
$$
is then $\bS_n$-equivariant for these actions. It is also $\langle\iota\rangle$-equivariant for its action on $\tilde{\calQ}$ given by

$$
\iota(X,g\bfT,\sigma):=(\iota(X),\iota(g)\bfT,\iota(\sigma)).
$$
As $\iota$ acts trivially on $\bS_n$ we get an action of $\bS'_n$ on $\tilde{\calQ}_n$ and  $p$ is $\bS_n'$-equivariant.
\bigskip

\begin{lemma}If the $\bfG_n$-conjugacy class of $\sigma\in\bm{\t}^{\rm gen}_n$ in $\bm{\g}_n$ is $\calC_1\times\cdots\times\calC_k$, the projection 
$$
\calQ_\sigma:=p^{-1}(\sigma)\rightarrow\calQ,\hspace{1cm}(X,g\bfT,\sigma)\mapsto X
$$
is an isomorphism.
\label{Q}\end{lemma}

For $\sigma\in\bm{\t}^{\rm gen}_n$ and $w'\in\bS'_n$, denote by $w':\calQ_\sigma\rightarrow\calQ_{w'\cdot\sigma}$ the isomorphism given by $(X,g\bfT_n,\sigma)\mapsto w'\cdot(X,g\bfT_n,\sigma)$.

\begin{theorem}\cite[Theorem 2.3]{HLRV2}Assume that $\K=\overline{\F}_q$ with ${\rm char}(\K)>>0$ or $\K=\C$ and let $\kappa$ be  $\overline{\Q}_\ell$ if $\K=\overline{\F}_q$ (with $\ell \nmid q$) and let $\kappa$ be $\C$ if $\K=\C$. The sheaf $R^ip_!\kappa$ is constant.\label{2.3}\end{theorem}

Therefore, for any $\sigma,\tau\in\bm{\t}^{\rm gen}_n$, there exists a canonical isomorphism $i_{\sigma,\tau}:H_c^i(\calQ_\sigma)\rightarrow H_c^i(\calQ_\tau)$ such that
$$
i_{\sigma,\tau}\circ i_{\zeta,\sigma}=i_{\zeta,\tau}
$$
for all $\sigma,\tau,\zeta\in\bm{\t}^{\rm gen}_n$.

Since $p$ is $\bS'_n$-equivariant, the isomorphisms $i_{\sigma,\tau}$ are compatible with the action of $\bS'_n$.
\bigskip

We define a representation
$$
\rho^j:\bS'_n\rightarrow \GL\left(H_c^{2j}(\calQ_\sigma)\right)
$$
by $\rho^j(w')=i_{w'\cdot \sigma,\sigma}\circ(w'{^{-1}})^*$.
Thanks to Lemma \ref{Q}, we get an action of $\bS'_n$ on $H_c^i(\calQ)$.

\subsection{Multiplicities and quiver varieties}\label{MQ}

{\bf Preliminaries}
\bigskip

\noindent For a partition $\mu$ of $n$ we denote by $M_\mu$ an irreducible $\overline{\Q}_\ell[S_n]$-module corresponding to $\mu$. For a type $\omega=\{(d_i,\omega^i)^{m_i}\}_{i=1,\dots,r}\in\T_n$, we consider the subgroup

$$
S_\omega=\prod_i\underbrace{(S_{|\omega^i|})^{d_i}\times\cdots\times(S_{|\omega^i|})^{d_i}}_{m_i}
$$
of $S_n$ and the $S_\omega$-module

$$
M_\omega:=\bigotimes_{i=1}^r(\underbrace{T^{d_i}M_{\omega^i}\otimes\cdots\otimes T^{d_i}M_{\omega^i}}_{m_i})
$$
where $T^dV$ stands for $V\otimes\cdots \otimes V$ ($d$ times).

The permutation action of $S_{d_i}$ on the factors of $(S_{|\omega^i|})^{d_i}$ and $T^{d_i}M_{\omega^i}$ induces an action of $\prod_i(S_{d_i})^{m_i}$ on both $S_\omega$ and $M_\omega$ and so we get an action of $S_\omega\rtimes\prod_i(S_{d_i})^{m_i}$ on $M_\omega$.

We may regard $S_\omega\rtimes \prod_i(S_{d_i})^{m_i}$ as a subgroup of the normalizer $N_{S_n}(S_\omega)$. Any $S_n$-module becomes thus an $S_\omega\rtimes\prod_i(S_{d_i})^{m_i}$-module by restriction. 

Now let $N$ be any $S_n$-module, we get an action of $\prod_i(S_{d_i})^{m_i}$ on 
$$
{\rm Hom}_{S_\omega}\left(M_\omega, N\right)
$$
(where $N$ is considered as an $S_\omega\rtimes\prod_i(S_{d_i})^{m_i}$-module by restriction) as 
$$
(r\cdot f)(v)=r\cdot(f(r^{-1}\cdot v))
$$
for any $f\in{\rm Hom}_{S_\omega}(M_\omega,N)$ and $r\in\prod_i(S_{d_i})^{m_i}$.

 Let $v_\omega$ be the element of $\prod_i(S_{d_i})^{m_i}$ whose coordinates act by circular permutation of the factors on each $T^{d_i}M_{\omega^i}$ and put
$$
c_\omega(N):={\Tr}\left(v_\omega \left| \,{\rm Hom}_{S_\omega}\left(M_\omega,N\right)\right.\right).
$$

\begin{lemma}(1) The function $s_\omega$ decomposes into the following sum of Schur functions as
$$
s_\omega=\sum_{\mu\in\calP_n}c_\omega(M_\mu)s_\mu.
$$
(2) We have

$$
c_\omega(M_{\mu^*})=(-1)^{r(\omega)}c_{\omega^*}(M_\mu).
$$
\label{decomp}\end{lemma}

\begin{proof}The first assertion is \cite[Proposition 6.2.5]{L1}. Let
  us prove the second assertion. To alleviate the notation, we assume
  (without loss of generality) that all $m_i=1$
  i.e. $\omega=\{(d_i,\omega^i)\}_{i=1,\dots r}$. By \cite[Proposition
  6.2.4]{L1} we have
$$
c_\omega(M_\mu)=\sum_\rho\chi^\mu_\rho\sum_\alpha\left(\prod_{i=1}^rz_{\alpha^i}^{-1}\chi^{\omega^i}_{\alpha^i}\right)
$$
where the second sum runs over all
$\alpha=(\alpha^1,\dots,\alpha^r)\in\calP_{|\omega^1|}\times\cdots\times\calP_{|\omega^r|}$
such that $\bigcup_ id_i\cdot\alpha^i=\rho$ (recall that $d\cdot\mu$
is the partition obtained from $\mu$ by multiplying all parts of $\mu$
by $d$).

Using that $\chi^{\mu^*}=\varepsilon\otimes\chi^\mu$ where $\varepsilon$ is the sign character, we are reduced to proving the following identity

\begin{equation}
\varepsilon(\rho)=(-1)^{n+\sum_i|\alpha^i|}\prod_{i=1}^r\varepsilon(\alpha^i)
\label{var}\end{equation}
whenever $\bigcup_ id_i\cdot\alpha^i=\rho$.

We have

$$
\varepsilon(\rho)=\prod_ i\varepsilon(d_ i\cdot\alpha^i).
$$
Since $n=\sum_ id_ i|\alpha^i|$ the identity (\ref{var}) is a consequence of the following identity

$$
\varepsilon(d\cdot\lambda)=(-1)^{(d+1)|\lambda|}\varepsilon(\lambda)
$$
where $d$ is a positive integer and $\lambda$ a partition.
\end{proof}

\bigskip

{\bf Main result}

\bigskip

\noindent We can generalize this to a multi-type $\omhat=(\omega_1,\dots,\omega_k)$ with all $\omega_i$ of same size $n$, by replacing $S_\omega$, $M_\omega$ and $v_\omega$ by

$$
S_\omhat:=S_{\omega_1}\times\cdots\times S_{\omega_k},\hspace{.5cm} M_\omhat:=M_{\omega_1}\boxtimes\cdots\boxtimes M_{\omega_k},\hspace{.5cm}v_\omhat=(v_{\omega_1},\dots,v_{\omega_k})
$$
and for any $\bS_n$-module $N$ we define 

$$
c_\omhat(N):=\Tr\left(v_\omhat\,|\, {\rm Hom}_{S_\omhat}(M_\omhat,N)\right).
$$ 

\begin{remark}
If $N$ is of the form $N_1\boxtimes\cdots\boxtimes N_k$ with $N_i$ any $S_n$-module, then

$$
c_\omhat(N)=c_{\omega_1}(N_1)\cdots c_{\omega_k}(N_k).
$$
\end{remark}

Now let $N$ be an $\bS'_n$-module. We extend
trivially the action of $N_{\bS_n}(S_\omhat)$ on $M_\omhat$ to an
action of
$N_{\bS_n'}(S_\omhat)=N_{\bS_n}(S_\omhat)\times\langle \iota\rangle$
on $M_\omhat$. We thus get an action of
$N_{\bS_n'}(S_\omhat)/S_\omhat=\left(N_{\bS_n}(S_\omhat)/S_\omhat\right)\times\langle
\iota\rangle$ on ${\rm Hom}_{S_\omhat}\left(M_\omhat,N\right)$, and we
define

$$
c'_\omhat(N):={\Tr}\left(v_\omhat \,\iota\left| {\rm Hom}_{S_\omhat}(M_\omhat,N)\right.\right).
$$

Let $\calQ_n$ be the quiver variety defined in \S \ref{quiver} and let $\bM^\bullet_n$ be the graded $\bS'_n$-module defined by

$$
\bM^i_n=H_c^{2i+d}(\calQ_n)\otimes (\varepsilon^{\boxtimes k})
$$
where $\varepsilon ^{\boxtimes k}=\varepsilon\boxtimes\cdots\boxtimes\varepsilon$ with $\varepsilon$ the sign representation of $S_n$.

\begin{theorem}Let $\omhat\in(\T_n)^k$. 

(1) We have
$$
V_\omhat(q)=(-1)^{r(\omhat)}\sum_ic_{\omhat}(\bM^i_n)\, q^i.
$$

(2) We have
$$
V'_\omhat(q)=(-1)^{n(\omhat^*)+r(\omhat)}\sum_ic'_{\omhat}(\bM^i_n)\, q^i.
$$
\label{maintheo}\end{theorem}
\bigskip

From the above theorem and Theorem \ref{genmulti}(2) we have
$$
\H_\muhat(-q)=(-1)^{r'(\muhat)+r(\muhat)+n+1}\sum_ic'_\muhat(\bM^i_n)\, q^i
$$
and from Theorem \ref{genmulti}(1)  we also have
\begin{equation}
\H_\muhat(q)=\sum_ic_\muhat(\bM^i_n)\, q^i
\label{A}\end{equation}
from which we deduce the following formula as 
$$
r(\muhat)+r'(\muhat)\equiv k(\lceil n/2\rceil+n)\mod 2.
$$

\begin{corollary}
$$
c'_\muhat(\bM^i_n)=(-1)^{i+k(\lceil n/2\rceil +n)+n+1}c_{\muhat}(\bM^i_n).
$$
\end{corollary}

%
%

\section{Proof of Theorem \ref{maintheo}}

When $\omhat$ is a multi-partition, the assertion (1) is proved in \cite[End of proof of Theorem 2.3]{L2}.  We will prove the assertion (2) first in the case of multi-partitions and then deduce from it the case of arbitrary multi-types. The reduction of the general case to the multi-partition case is completely similar for $G^F$ and $G^{F'}$.

\subsection{Quiver varieties and Fourier transforms}

In this section, $\K=\overline{\F}_q$, $G=\GL_n(\K)$ and $\frak{g}=\gl_n(\K)$. We denote by $F:\g\rightarrow\g$ the standard Frobenius that raises matrix coefficients to their $q$-th power. We also denote by $F':\g\rightarrow\g$, $X\mapsto -{^t}F(X)$.

The conjugation action of $G$ on $\g$ is compatible with both
Frobenius endomorphisms $F$ and $F'$, i.e. 
$$
F(gXg^{-1})=F(g)F(X)F(g^{-1}),\hspace{.5cm}F'(gXg^{-1})=F'(g)F'(X)F'(g^{-1})
$$
for any $g\in G$ and $X\in\g$, and so $G^F$ (resp. $G^{F'}$) acts on $\g^F$ (resp. $\g^{F'}$).
\bigskip

\begin{nothing}\emph{Quiver variety.} \label{QV}Since for all $x\in\g$, the stabilizer $C_G(x)$ is connected, the set of  $G^F$-orbit of $\g^F$  (resp. the set of $G^{F'}$-orbits of $\g^{F'}$) is naturally in bijection with the set of $F$-stable (resp. $F'$-stable) $G$-orbits of $\g$, i.e. if $\calO$ is a $G$-orbit of $\g$ stable by the Frobenius, then any two rational elements of $\calO$ are rationnally conjugate.
\bigskip

Denote by $\tilde{\Xi}$ (resp. $\tilde{\Xi}'$) the set of $F$-orbits (resp. $F'$-orbits) of $\K$. 
\bigskip

 Analogously to conjugacy classes of $G^F$ and $G^{F'}$, the set of $F$-stable (resp. $F'$-stable) $G$-orbits of $\g$ is in bijection with the set $\calP_n(\tilde{\Xi})$ (resp. $\calP_n(\tilde{\Xi}')$) of all maps $f:\tilde{\Xi}\rightarrow\calP$ (resp. $f:\tilde{\Xi}'\rightarrow\calP$) such that
 
$$
|f|:=\sum_\xi |\xi| \,|f(\xi)|=n
$$
where $|\xi| $ denotes the size of the orbit $\xi$.
 \bigskip
 
 As for conjugacy classes, we can associated to any $f\in\calP_n(\tilde{\Xi})$ (resp. $f\in \calP_n(\tilde{\Xi}')$) a type $\t(f)\in\T_n$.
 \bigskip
 
 The types of the $F'$-stable semisimple regular $G$-orbits of $\g$ are of the form $\{(d_i,1)^{m_i}\}$ with 
 
 $$
 \sum_id_im_i=n,
 $$
 and are therefore parametrized by the partitions of $n$ and so by the conjugacy classes of $S_n$ : the partition of $n$ corresponding to $\{(d_i,1)^{m_i}\}_i$  is

 $$
 \sum_i \underbrace{d_i+\cdots+d_i}_{m_i}
 $$

 \bigskip
 
 For example, the types $(1,1)^2$ and $(2,1)$ are the types of the orbits of 
 
 $$
 \left(\begin{array}{cc}a&0\\0&b\end{array}\right),\hspace{1cm}\left(\begin{array}{cc}x&0\\0&-x^q\end{array}\right),
 $$
 where $a\neq b\in\{z^q=-z\}$, and $x\in\F_{q^2}\backslash\{z^q=-z\}$, corresponding respectively to the trivial and non-trivial element of $S_2$.
\bigskip

For short we will say that an $F'$-stable semisimple regular $G$-orbit of $\g$ is of type $w\in S_n$ if its type corresponds to the conjugacy class of $w$ in $S_n$.

\bigskip

For a $k$-tuple $\w=(w_1,\dots,w_k)\in\bS_n$, we choose a generic $k$-tuple $\calC^\w=(\calC^{w_1}\dots,\calC^{w_k})$ of $F'$-stable semisimple regular $G$-orbit of $\g$ of type $\w$ and we consider the associated quiver variety

$$
\calQ^\w:=\calV^\w/\!/\PGL_n
$$
where
$$
\calV^\w:=\left\{(X_1,\dots,X_k)\in\calC^{w_1}\times\cdots\times\calC^{w_k}\,\left|\, \sum_iX_i=0\right\}\right.
$$
\end{nothing}
 
\begin{nothing}\emph{Introducing Fourier transforms.} For the definition and properties of these Fourier transforms we follow G. Lehrer \cite{Lehrer}. 

Denote by $\calC(\g^{F'})$ the $\overline{\Q}_\ell$-vector space of functions $\g^{F'}\rightarrow\overline{\Q}_\ell$ constant on $G^{F'}$-orbits which we equip with $\langle\,,\,\rangle$ defined by

$$
\langle f_1,f_2\rangle_{\g^{F'}}=\frac{1}{|G^{F'}|}\sum_{x\in\g^{F'}}f_1(x)\overline{f_2(x)},
$$
for any $f_1,f_2\in\calC(\g^{F'})$ where $\overline{\Q}_\ell\rightarrow\overline{\Q}_\ell, x\mapsto \overline{x}$ is the involution corresponding to the complex conjugation under an isomorphism $\overline{\Q}_\ell\simeq\C$ we have fixed.

Fix a non-trivial additive character $\psi:\F_q\rightarrow\overline{\Q}_\ell$. Notice that the trace map $\Tr$ on $\g$ satisfies

$$
\Tr(F'(x)F'(y))=\Tr(xy)^q.
$$
for all $x,y\in\g$. Define the Fourier transform $\calF^\g:\calC(\g^{F'})\rightarrow\calC(\g^{F'})$ by

$$
\calF^\g(f)(y)=\sum_{x\in\g^{F'}}\psi(\Tr(yx))f(x)
$$
for any $y\in\g^{F'}$ and $f\in\calC(\g^{F'})$.

Consider the convolution product $*$ on $\calC(\g^{F'})$ defined by 

$$
(f_1*f_2)(x)=\sum_{y+z=x}f_1(y)f_2(z),
$$
for $x\in\g^{F'}$, $f_1,f_2\in\calC(\g^{F'})$.

We have the following straightforward proposition.

\begin{proposition}(1) We have

$$
\calF^\g(f_1*f_2)=\calF^\g(f_1)\calF^\g(f_2)
$$
for all $f_1,f_2\in\calC(\g^{F'})$.

\noindent (2) For $f\in\calC(\g^{F'})$ we have

$$
|\g^{F'}|\cdot f(0)=\sum_{x\in\g^{F'}}\calF^\g(f)(x).
$$
\end{proposition}

For a $G^{F'}$-orbit $O$ of $\g^{F'}$, let $1_O\in\calC(\g^{F'})$ denote the characteristic function of $O$, i.e.

$$
1_O(x)=\begin{cases}1&\text{ if } x\in O\\0&\text{ otherwise.}\end{cases}
$$

\begin{proposition}We have

$$
|(\calQ^\w)^{F'}|=\frac{(q+1)}{|\g^{F'}|}\left\langle\prod_{i=1}^k\calF^\g\left(1_{(\calC^{w_i})^{F'}}\right),1\right\rangle_{\g^{F'}}
$$

\label{quivercount}\end{proposition}

\begin{proof}Since $\PGL_n(\K)$ is connected and acts freely on $\calV^\w$, we have

$$
|(\calQ^\w)^{F'}|=\frac{|(\calV^\w)^{F'}|}{|\PGL_n(\K)^{F'}|}=\frac{(q+1)|(\calV^\w)^{F'}|}{|\GL_n(\K)^{F'}|}.
$$

On the other hand

\begin{align*}
|(\calV^\w)^{F'}|&=\#\left\{(X_1,\dots,X_k)\in(\calC^{w_1})^{F'}\times\cdots\times(\calC^{w_k})^{F'}\left|\sum_iX_i=0\right\}\right.\\
&=\left(1_{(\calC^{w_1})^{F'}}*\cdots* 1_{(\calC^{w_k})^{F'}}\right)(0)\\
&=\frac{1}{|\g^{F'}|}\sum_{x\in\g^{F'}}\prod_{i=1}^k\calF^\g\left(1_{(\calC^{w_i})^{F'}}\right)(x).
\end{align*}
\end{proof}
\end{nothing}

\subsection{Fourier transforms and irreducible characters: Springer's theory}

Consider a type of the form  $\omega=\{(d_i,1)^{m_i}\}_{i=1,\dots,r}\in\T_n$ (we call types of this form \emph{regular semisimple}), and denote by

$$
T_\omega^{F'}=\prod_{i, \,d_i \text{ even}}\GL_1(\F_{q^{d_i}})^{m_i}\prod_{i, \,d_i\text{ odd}}\GU_1(\F_{q^{d_i}})^{m_i}
$$
its associated rational maximal torus. 

An  irreducible character $\calX_f$ of $G^{F'}$ of type  $\t(f)=\omega$ is called \emph{regular semisimple}. 

We have
\begin{equation}
\calX_f=(-1)^{r(\omega)}R_{T_\omega^{F'}}^{G^{F'}}(\theta_f)
\label{DLfor}\end{equation}
for some linear character $\theta_f$ of $T_\omega^{F'}$ (see Theorem \ref{LS}).

Moreover, for all $g\in G^{F'}$ with Jordan decomposition $g=g_sg_u$,  we have the following character formula \cite[Theorem 4.2]{DL}

\begin{equation}
R_{T_\omega^{F'}}^{G^{F'}}(\theta_f)(g)=\frac{1}{|C_G(g_s)^{F'}|}\sum_{\{h\in G^{F'}\,|\, g_s\in hT_\omega h^{-1}\}}Q_{hT_\omega^{F'} h^{-1}}^{C_G(g_s)^{F'}}(g_u)\theta_f(h^{-1}g_s h)
\label{charfor}\end{equation}
where 
$$
Q_{hT_\omega^{F'} h^{-1}}^{C_G(g_s)^{F'}}:=R_{hT_\omega^{F'} h^{-1}}^{C_G(g_s)^{F'}}(1_{\{1\}})
$$
is the so-called \emph{Green function} defined by Deligne-Lusztig \cite{DL}. 

Denote by $\t_\omega$ the Lie algebra of $T_\omega$. In \cite{L0'}, we defined a Lie algebra version of Deligne-Lusztig induction, namely we defined a $\overline{\Q}_\ell$-linear map
$$
R_{\t_\omega^{F'}}^{\g^{F'}}:\calC(\t_\omega^{F'})\rightarrow\calC(\g^{F'})
$$
by the same formula as (\ref{charfor}), i.e.
\begin{equation}
R_{\t_\omega^{F'}}^{\g^{F'}}(\eta)(x)=\frac{1}{|C_G(x_s)^{F'}|}\sum_{\{h\in G^{F'}\,|\, x_s\in h\t_\omega h^{-1}\}}Q_{h\t_\omega^{F'} h^{-1}}^{C_\g(x_s)^{F'}}(x_n)\, \eta(h^{-1}g_s h)
\label{DL-Let}\end{equation}
for $x\in\g^{F'}$ with Jordan decomposition $x=x_s+x_n$ and where 
$$
Q_{h\t_\omega^{F'} h^{-1}}^{C_\g(x_s)^{F'}}(x_n):=Q_{hT_\omega^{F'} h^{-1}}^{C_G(g_s)^{F'}}(x_n+1).
$$
We have the following special case of  \cite[Theorem 7.3.3]{L0}.

\begin{theorem}Let $\calC_h$ be a regular semisimple orbit of $\g^{F'}$ of type $\t(h)=\omega$, then 

$$
\calF^{\g^{F'}}(1_{\calC_h})=(-1)^{r'(\omega)}q^{\frac{n^2-n}{2}}R_{\t_\omega^{F'}}^{\g^{F'}}(\eta_h)
$$
where $\eta_h:\t_\omega^{F'}\rightarrow\overline{\Q}_\ell$, $z\mapsto \psi(\Tr(zx))$ with $x\in\t_\omega^{F'}$ a fixed representative of $\calC_h$ in $\t_\omega^{F'}$.
\label{5.3}\end{theorem}

The above formula shows that the computation of the values of
$\calF^{\g^{F'}}(1_{\calC_f})$ and $\calX_f$ is identical. This
connection between Fourier transforms and characters of finite
reductive groups was first observed and investigated by T. A. Springer
\cite{Spr}\cite{Spr1}\cite{K} and later by G. Lusztig \cite{Lus-Four} and G. I. Lehrer \cite{Lehrer}\cite{Lehrer1}. As a consequence we get the additive
version of Theorem \ref{sum-theo}(2).

\begin{theorem}Assume that $(\calC_1,\dots,\calC_k)$ is a generic tuple of $F'$-stable regular semisimple orbits of $\g^{F'}$ of type $\omhat=(\omega_1,\dots,\omega_k)$. Then for any type $\tau\in\T_n$ we have
\begin{equation}
\sum_{f\in\calP_n(\tilde{\Xi}'), \t(f)=\tau}\prod_{i=1}^k\calF^\g(1_{\calC_i^{F'}})(\calC'_f)=q^{\frac{k(n^2-n)+2}{2}}c^o_\tau (-1)^{r'(\omhat)}\prod_{i=1}^k \left\langle \tilde{H}_\tau(\x_i;-q),s_{\omega_i}\right\rangle,
\label{for}\end{equation}
where $\calC'_f$ denotes the $G^{F'}$-orbit of $\g^{F'}$ corresponding to $f$.
\label{sum'}\end{theorem}

\begin{proof} In the LHS of Formula (\ref{for}) we first replace $\calF^\g(1_{\calC_i^{F'}})$ by $R_{\t_{\omega_i}^{F'}}^{\g^{F'}}(\eta_i)$ (see Theorem \ref{5.3}).  Using Formula (\ref{DL-Let}),  the proof is completely similar to its multiplicative version (Theorem \ref{sum-theo}(2)). For more details see the proof of the analogous statement in the case of the Frobenius endomorphism $F$ \cite[Lemma 6.2.3]{HLRV}.

\end{proof}

\begin{theorem}Let $(\calX'_1,\dots,\calX'_k)$ be a generic $k$-tuple of regular semisimple  irreducible characters of $G^{F'}$ and let $(\calC_1,\dots,\calC_k)$ be a generic $k$-tuple of $F'$-stable regular semisimple orbits of $\g^{F'}$ of same type as $(\calX'_1,\dots,\calX'_k)$. Then

$$
\left\langle\calX'_1\otimes\cdots\otimes\calX'_k,1\right\rangle_{G^{F'}}=q^{-\frac{1}{2}{\rm dim}\,\calQ}\frac{(q+1)}{|\g^{F'}|}\left\langle\prod_i\calF^\g(1_{\calC_i^{F'}}),1\right\rangle_{\g^{F'}}
$$
\label{comparison}\end{theorem}

\begin{proof}The analogous formula in the case of the standard Frobenius $F$ instead of $F'$ is a particular case of \cite[Theorem 6.9.1]{L1} and the proof for $F'$ is completely similar.  However, since  the proof of [loc. cite] simplifies in the regular semisimple case, we give it for the convenience of the reader.

For each $i=1,\dots,k$, let $\omega_i$ be the common type of $\calX'_i$ and $\calC_i$. Then

\begin{align*}
\left\langle\prod_i\calF^\g(1_{\calC_i^{F'}}),1\right\rangle&=\frac{1}{|G^{F'}|}\sum_{x\in\g^{F'}}\prod_i\calF^\g(1_{\calC_i^{F'}})(x)\\
&=\sum_{f\in\calP_n(\tilde{\Xi}')}\frac{1}{a'_f(q)}\prod_i\calF^\g(1_{\calC_i^{F'}})(\calC'_f)
\end{align*}
where for $f\in\calP_n(\tilde{\Xi}')$, $\calC'_f$ is the associated $G^{F'}$-orbit of $\g^{F'}$ and $a'_f(q)$ the size of the stabilizer in $G^{F'}$ of an element of $\calC'_f$.

We thus have
\begin{align*}
\left\langle\prod_i\calF^\g(1_{\calC_i^{F'}}),1\right\rangle&=\sum_{\tau\in\T_n}\frac{1}{a'_\tau(q)}\sum_{f\in\calP_n(\tilde{\Xi}'),\t(f)=\tau}\prod_i\calF^\g(1_{\calC_i^{F'}})(\calC'_f)\\
&=q^{\frac{k(n^2-n)+2}{2}}(-1)^{r'(\omhat)+n}\sum_{\tau\in\T_n}\frac{1}{a_\tau(-q)}c_\tau^o\prod_{i=1}^k\left\langle\tilde{H}_\tau(\x_i;-q),s_{\omega_i}\right\rangle\\
&=\frac{q^{\frac{k(n^2-n)+2}{2}}(-1)^{r'(\omhat)+n+1}}{q+1}\H_\omhat(-q)\\
&=\frac{q^{\frac{k(n^2-n)+2}{2}}}{q+1}\langle\calX'_1\otimes\cdots\otimes\calX'_k,1\rangle.
\end{align*}
The second equality is a consequence of Theorem \ref{sum'} and the last equality follows from Theorem \ref{genmulti}(2)  and so Theorem \ref{comparison} follows from (\ref{dim}).

\end{proof}

\subsection{Proof of Theorem \ref{maintheo}}\label{5.4}
 
We first prove the theorem when each coordinate of $\omhat$ is a regular semisimple type. We then deduce the case where $\omhat$ is a multi-partition, i.e. each coordinate of $\omhat$ is of the form $(1,\mu)$ with $\mu$ a partition. We finally deduce the general case from the multi-partition case.
 
\begin{nothing}\emph{Semi-simple regular case.} We saw in \S \ref{QV}, that regular semisimple types in $\T_n$ are parametrized by the conjugacy classes of $S_n$. Assume that all coordinates of $\omhat=(\omega_1,\dots,\omega_k)$ are regular semisimple. The element $v_\omhat\in\bS_n$ defined in \S \ref{MQ}   is an element in the corresponding conjugacy class of $\bS_n$. 

Let $(\calX'_1,\dots,\calX'_k)$ be a $k$-tuple of irreducible characters of $G^{F'}$ of type $\omhat$. From Theorem \ref{comparison} and Proposition \ref{quivercount}, we get the following identity
$$
\langle \calX'_1\otimes\cdots\otimes\calX'_k,1\rangle_{G^{F'}}=q^{-\frac{{\rm dim}\, \calQ}{2}}|(\calQ^{v_\omhat})^{F'}|.
$$
On the other hand we can follow line by line the proof of \cite[Theorem 2.6]{HLRV2} to get the following one.

\begin{theorem}We have

$$
|(\calQ^{v_\omhat})^{F'}|=\sum_i\Tr\left(v_\omhat\iota\,|\, H_c^{2i}(\calQ)\right)\, q^i.
$$

\end{theorem}

As 

$$
\varepsilon^{\boxtimes k}(v_\omhat)=(-1)^{r(\omhat)}
$$
we have

\begin{align*}
|(\calQ^{v_\omhat})^{F'}|&=q^{\frac{{\rm dim}\, \calQ}{2}}(-1)^{r(\omhat)}\sum_i\Tr\left(v_\omhat\iota\,|\, \bM^i_n\right)\, q^i\\
&=q^{\frac{{\rm dim}\, \calQ}{2}}(-1)^{r(\omhat)}\sum_i c'_\omhat(\bM^i_n)\, q^i
\end{align*}
as $M_\omhat$ is trivial. We thus get Theorem \ref{maintheo} in the regular semisimple case as $n(\omhat^*)=0$.

\end{nothing}

\begin{nothing}\emph{Multi-partition case.} First of all notice that  if $\lambda$ is a partition 
$$
\underbrace{\lambda_1+\cdots+\lambda_1}_{m_1}+\underbrace{\lambda_2+\cdots+\lambda_2}_{m_2}+\cdots
$$
with $\lambda_i\neq\lambda_j$ for $i\neq j$, then 

$$
p_\lambda=s_\omega
$$
where $\omega$ is the regular semisimple type $\{(\lambda_i,1)^{m_i}\}$. In the following we will write $[\lambda]$ for the regular semisimple type associated to a partition $\lambda$. 

Assume now that $\omhat$ is a multi-partition
$\muhat=(\mu^1,\dots,\mu^k)$, i.e. the $i$-coordinate of $\omhat$ is
the type $(1,\mu^i)$. Decomposing Schur functions into power sums
functions $p_\lambda$ we get
$$
\H_\muhat(-q)=\sum_\lambdahat z_\lambdahat^{-1}\chi^\muhat_\lambdahat \H_{[\lambdahat]}(-q)
$$
Using the theorem for regular semisimple types together with Theorem \ref{genmulti}(2), we get
\begin{align*}
\H_\muhat(-q)&=\sum_\lambdahat z_\lambdahat^{-1}\chi^\muhat_\lambdahat (-1)^{r'([\lambdahat])+r([\lambdahat])+n+1}\sum_ic'_{[\lambdahat]}(\bM^i_n) q^i\\
&=(-1)^{n+1}\sum_i\left(\sum_\lambdahat z_\lambdahat^{-1}\chi^\muhat_\lambdahat (-1)^{r'([\lambdahat])+r([\lambdahat])}\Tr\left(v_{[\lambdahat]}\,\iota\,| \,\bM^i_n\right)\right)\, q^i
\end{align*}
Therefore 
$$
(-1)^{r'(\muhat)+n(\muhat^*)+n+1}\H_\muhat(-q)=(-1)^{n(\muhat^*)}\sum_i\left(\sum_\lambdahat z_\lambdahat^{-1}\chi^\muhat_\lambdahat (-1)^{r'([\lambdahat])+r'(\muhat)+r([\lambdahat])}\Tr\left(v_{[\lambdahat]}\,\iota\,|\, \bM^i_n\right)\right)\, q^i
$$
However, 

$$
(-1)^{r'(\muhat^*)+r'([\lambdahat])}=(-1)^{r([\lambdahat])}.
$$
and so 
\begin{align*}
(-1)^{r'(\muhat)+n(\muhat^*)+n+1}\H_\muhat(-q)&=(-1)^{n(\muhat^*)}\sum_i\left(\sum_\lambdahat z_\lambdahat^{-1}\,\chi^{\muhat}_\lambdahat \,\Tr\left(v_{[\lambdahat]}\,\iota\,|\, \bM^i_n\right)\right)\, q^i\\
&=(-1)^{n(\muhat^*)}\sum_i\Tr\left(\iota \,\left| {\rm Hom}_{\bS_n}(M_{\muhat},\bM^i_n)\right.\right)\, q^i\\
&=(-1)^{n(\muhat^*)}\sum_ic'_{\muhat}(\bM^i_n)\, q^i
\end{align*}
hence the result for multi-partitions by Theorem \ref{genmulti}(2) as $r(\muhat)$ is even.
\end{nothing}

\begin{nothing}\emph{General case.} Assume now that $\omhat\in(\T_n)^k$ is arbitrary. By Lemma \ref{decomp} we have

\begin{align*}
\H_\omhat(-q)&=\sum_{\muhat \in(\calP_n)^k}c_\omhat(M_\muhat)\H_\muhat(-q)\\
&=\sum_\muhat c_\omhat(M_\muhat)(-1)^{r'(\muhat)+n+1}\sum_ic_{\muhat}'(\bM^i_n)q^i\\
&=(-1)^{n+1}\sum_i\sum_\muhat (-1)^{r'(\muhat)} c_\omhat(M_{\muhat})c_\muhat'(\bM^i_n)q^i
\end{align*}
We thus have

\begin{align*}(-1)^{r'(\omhat)+n(\omhat^*)+n+1}\H_\omhat(-q)&=(-1)^{n(\omhat')+r'(\omhat)}\sum_i\sum_\muhat (-1)^{r'(\muhat)}c_\omhat(M_\muhat)c'_\muhat(\bM^i_n) q^i\\
&=(-1)^{n(\omhat')+r(\omhat)}\sum_i\sum_\muhat c_\omhat(M_\muhat)c'_\muhat(\bM^i_n) q^i
\end{align*}
since 

$$
r'(\muhat)+r'(\omhat)\equiv r(\omhat)\mod 2.
$$

By Theorem \ref{genmulti}(2), to complete the proof of Theorem \ref{maintheo} we are thus reduced to proving the identity

\begin{equation}
\sum_\muhat c_\omhat(M_\muhat)c'_\muhat(\bM^i_n)=c'_\omhat(\bM^i_n).
\label{tech}\end{equation}
The $\bS_n'$-module $\bM^i_n$ decomposes as

$$
\bM^i_n=\bigoplus_{\muhat\in(\calP_n)^k} {\rm Hom}_{\bS_n}(M_\muhat,\bM^i_n)\otimes M_\muhat
$$
where $\bS_n$ acts on $M_\muhat$ and $\langle\iota\rangle$ acts on
${\rm Hom}_{\bS_n}(M_\muhat,\bM^i_n)$. Hence
$$
{\rm Hom}_{S_\omhat}(M_\omhat,\bM^i_n)\simeq \bigoplus_\muhat\left({\rm Hom}_{S_\omhat}(M_\omhat,M_\muhat)\otimes_{\overline{\Q}_\ell}{\rm Hom}_{\bS_n}(M_\muhat,\bM^i_n)\right).
$$
and the action of $v_{\omhat^*}\,\iota$ on the left corresponds to $v_{\omhat^*}\otimes\iota$ on the right, hence the identity (\ref{tech}).
\end{nothing}

\section{Module theoretic interpretation of the polynomials $\Tau_{{\bf \mu}}(u,q)$}

\subsection{Exp of graded modules}

Assume given a module 

$$
\bH^\bullet=\bigoplus_{n\geq1}\bH^\bullet_n
$$
where $\bH^\bullet_n$ is a $q$-graded finite-dimensional $\bS_n$-module and denote by 

$$
{\rm ch}(\bH^\bullet):=\sum_{n\geq 1}\sum_{\muhat\in(\calP_n)^k}\sum_ic_\muhat(\bH^i_n)q^is_\muhat T^n
$$
its $q$-graded Frobenius characteristic function. 

For each $n>0$ define the $q$-graded $\bS_n$-module 

\begin{equation}
\widetilde{\bH}^\bullet_n:=\bigoplus_{\lambda\in\calP_n}{\rm Ind}_{{\bf N}_\lambda}^{\bS_n}\left(\bH^\bullet_\lambda\right)
\label{H}\end{equation}
where for a partition $\lambda=(1^{r_1},2^{r_2},\dots)$ of $n$ we put

$$
{\bf N}_\lambda:=\left(\prod_i(\bS_i)^{r_i}\right)\rtimes \prod_iS_{r_i},\hspace{1cm}\bH^\bullet_\lambda:=\boxtimes_i\left(\bH^\bullet_i\right)^{\boxtimes r_i}
$$
and $S_{r_i}$ acts by permutation of the coordinates on $(\bS_i)^{r_i}$ and $(\bH^\bullet_i)^{\boxtimes r_i}$. 

Notice that ${\bf N}_\lambda$ can be seen as a subgroup of the normalizer of $\prod_i(\bS_i)^{r_i}$ in $\bS_n$ (and so is a subgroup of $\bS_n$).
\bigskip

Following Getzler \cite{G} we prove the following result.

\begin{theorem}Put

$$
\Exp(\bH^\bullet):=\bigoplus_{n\geq 0}\widetilde{\bH}^\bullet_n.
$$
Then
$$
{\rm ch}\left(\Exp(\H^\bullet)\right)=\Exp({\rm ch}(\bH^\bullet)).
$$
\label{G}\end{theorem}

%
%
%

Now let $\mathcal{L}$ be the non-trivial irreducible module of $\Z/2\Z=\langle\iota\rangle$ and define the $q$-graded $\bS'_n$-module ${\bf H}^\bullet_n$ as

$$
{\bf H}^\bullet=\mathcal{L}\boxtimes \bH^\bullet.
$$

%

Extend the definition of the $q$-graded Frobenius characteristic map ${\rm ch}$ to $\bS'_n$-modules by mapping the irreducible modules $\mathcal{L}\boxtimes H_\muhat$ to $us_\muhat$.

Then

\begin{equation}
{\rm ch}({\bf H}^\bullet)=u\,{\rm ch}(\bH^\bullet).
\label{ch}\end{equation}

Replacing $\bH^\bullet_\lambda$ by ${\bf H}_\lambda^\bullet$ in (\ref{H}) we get

\begin{align*}
\widetilde{{\bf H}}^\bullet_n:&=\bigoplus_{\lambda\in\calP_n}{\rm Ind}_{{\bf N}_\lambda}^{\bS_n}\left({\bf H}^\bullet_\lambda\right)\\
&=\bigoplus_{\lambda\in\calP_n}\mathcal{L}^{\ell(\lambda)}\boxtimes{\rm Ind}_{{\bf N}_\lambda}^{\bS_n}\left(\bH^\bullet_\lambda\right)
\end{align*}

Put

$$
\Exp({\bf H}^\bullet):=\bigoplus_{n\geq 0}\widetilde{\bf H}^\bullet_n.
$$
Then

$$
{\rm ch}(\Exp({\bf H}^\bullet))=\sum_{n\geq 0}\sum_{\muhat\in(\calP_n)^k}\sum_{\lambda\in\calP_n}\sum_ iu^{\ell(\lambda)}c_\muhat\left({\rm Ind}_{{\bf N}_\lambda}^{\bS_n}(\bH^i_\lambda)\right)q^is_\muhat T^n.
$$
Theorem \ref{G} extends as

\begin{equation}
{\rm ch}(\Exp({{\bf H}}^\bullet))=\Exp({\rm ch}({\bf H}^\bullet)).
\label{G'}\end{equation}

\subsection{Module theoretic interpretation of the unipotent multiplicities} 

In this section we apply the results of the above section with $\bH^\bullet=\bM^\bullet$.

\begin{theorem}We have
$$
{\rm ch}(\Exp({\bf M}^\bullet))=1+u\sum_{n>0}\sum_{\muhat\in(\calP_n)^k}\Tau_\muhat(u,q) s_\muhat T^n
$$
and so

\begin{equation}
\Tau_\muhat(u,q)=\sum_{\lambda\in\calP_n}\sum_i u^{\ell(\lambda)-1}c_\muhat\left({\rm Ind}_{{\bf N}_\lambda}^{\bS_n}(\bM^i_\lambda)\right)q^i.
\label{tau-for}\end{equation}
In particular the polynomials $\Tau_\muhat(u,q)$ have non-negative integer coefficients.
\label{prop-tau}\end{theorem}

\begin{proof}Applying $\log$ to Formula (\ref{tau}) we get

$$
\sum_{d\geq 1}\Phi_d(u,q)\log\left(\Omega(\x_1^d,\dots,\x_k^d,q^d;T^d)\right)=\log\left(1+u\sum_{n>0}\sum_{\muhat\in(\calP_n)^k}\Tau_\muhat(u,q)s_\muhat T^n\right)
$$
We apply Lemma \ref{moz} with $h=u(q-1)$ so that $h_d=\Phi_d(u,q)$, and we deduce that
$$
u(q-1)\Log\left(\Omega(q)\right)=\Log\left(1+u\sum_{n>0}\sum_{\muhat\in(\calP_n)^k}\Tau_\muhat(u,q)s_\muhat T^n\right)
$$
and so 

$$
1+u\sum_{n>0}\sum_{\muhat\in(\calP_n)^k}\Tau_\muhat(u,q)s_\muhat T^n=\Exp\left(u(q-1)\Log\,\Omega(q)\right).
$$
The theorem is thus a consequence of Formula (\ref{G'}) together with the following theorem.
\end{proof}

\begin{theorem}We have

$$
\ch\left({\bf M}^\bullet\right)=u(q-1)\Log\,\Omega(q).
$$
\end{theorem}

\begin{proof} By Formula (\ref{ch}) we are reduced to proving that

\begin{equation}
\ch(\bM^\bullet)=(q-1)\Log\, \Omega(q).
\label{B}\end{equation}
We have

$$
\ch(\bM^\bullet)=\sum_{n\geq 1}\sum_{\muhat\in(\calP_n)^k}\sum_ic_\muhat(\bM^i_n)q^is_\muhat T^n
$$
and so Formula (\ref{B}) follows from Formula (\ref{A}).
\end{proof}

From Theorem \ref{prop-tau} together with Theorem \ref{EnnolaII} we deduce the following.

\begin{theorem}For any multi-partition $\muhat\in(\calP_n)^k$ we have

$$
U_\muhat(q)=\sum_{\lambda\in\calP_n}\sum_i c_\muhat\left({\rm Ind}_{{\bf N}_\lambda}^{\bS_n}(\bM^i_\lambda)\right)q^i,\hspace{1cm}(-1)^{\frac{1}{2}d_\muhat+n}U'_\muhat(q)=\sum_{\lambda\in\calP_n}\sum_i (-1)^{\ell(\lambda)+i-1}c_\muhat\left({\rm Ind}_{{\bf N}_\lambda}^{\bS_n}(\bM^i_\lambda)\right)q^i.
$$

\end{theorem}

\subsection{Proof of Theorem \ref{K-coeff}}\label{PK}

The constant term in $u$ in (\ref{tau-for}) corresponds to the partition $\lambda=(n^1)$  and 
$$
{\rm Ind}_{{\bf N}_{(n^1)}}^{\bS_n}(\bM^\bullet_{(n^1)})=\bM^\bullet_n.
$$
The assertion (i) follows thus from Proposition \ref{prop-tau} together with Theorem \ref{maintheo}.

The term of degree $n-1$ in $u$ in $\Tau_\muhat(u,q)$ corresponds to the longest partition $\lambda=(1^n)$. In this case $\bM^\bullet_\lambda$ is the trivial module of ${\bf N}_{(1^n)}\simeq S_n$ (embedded diagonally in $\bS_n$) and so  $c_\muhat\left({\rm Ind}_{{\bf N}_{(1^n)}}^{\bS_n}(\bM^\bullet_{(1^n)})\right)$ is the Kronecker coefficient $\langle\chi^{\mu^1}\otimes\cdots\otimes\chi^{\mu^k},1\rangle_{S_n}$ where $(\mu^1,\dots,\mu^k)=\muhat$.

\section{Examples}
\label{examples}
In this section we give a few explicit values for the polynomials
$V_\muhat(q),V_\muhat'(q),U_\muhat(q),U_\muhat'(q)$  for small values of $n$. Note that of the first two we
only need to list $V_\muhat(q)$ since we easily obtain $V'_\muhat(q)$
by Ennola duality (see Corollary \ref{Ennola-gen}).  To compute these polynomials we
implement in PARI-GP~\cite{PARI-GP} the infinite
products~\eqref{inf-prod-GL} and~\eqref{inf-prod-U} involving the
series $\Omega(\x,q;T)$ (here $\x$ stands collectively for the $k$ set
of infinite variables $(\x_1,\ldots,\x_k)$). The series
$\Omega(\x,q;T)$ itself was computed using code in Sage~\cite{sage}
written by A.~Mellit. The values we obtain for
$U_\muhat(q), U_\muhat'(q)$ match those in the tables in~\cite{mattig}
(but see Remark~\ref{sign-issue} below).

Concretely, define the rational functions $R_n(\x,q)\in \Lambda$ via
the expansion
$$
\log \Omega(\x,q;T)=\sum_{n\geq 1} R_n(\x,q)T^n.
$$
Then by~\eqref{inf-prod-GL} and~\eqref{inf-prod-U} we have 
\begin{equation}
\log\left(1+\sum_{n>0}\sum_{\muhat\in(\calP_n)^k}U_\muhat(q)s_\muhat T^n\right)
=\sum_{n\geq 1} \sum_{d\mid n} \Phi_d(q)R_{n/d}(\x^d,q^d)T^n
\end{equation}
and
\begin{equation}
\begin{split}
\log\left(1+\sum_{n>0}\sum_{\muhat\in(\calP_n)^k}U'_\muhat(q)s_\muhat T^n\right)
&=\sum_{n\geq 1} \sum_{d\mid
  n}(-1)^{n/d}\Phi'_d(q)R_{n/d}(\x^d,-q^d)T^n\\
&+
\sum_{n\geq 1} \sum_{d\mid n}
\Phi'_{2d}(q) R_{n/d}(\x^{2d},q^{2d})T^{2n}\\
&-\sum_{d\mid n}(-1)^{n/d}\Phi'_{2d}(q)R_{n/d}(\x^{2d},-q^{2d})T^{2n}
\end{split}
\end{equation}

\begin{remark}
\label{sign-issue}
As L\"ubeck points out the polynomials $U'_\muhat(q)$ do not
in general have non-negative coefficients. However, their values at
powers of primes must be non-negative as they give multiplicities of
tensor product of characters of a finite group. Hence, the coefficient of the highest power of $q$ must be positive.
\end{remark}

\begin{center}
\begin{tabular}{LLL|L}
\mu^1&\mu^2&\mu^3& V_\muhat\\
\hline
(1^2)&  (1^2)&  (1^2)&   1\\
\hline
(1^3)&  (1^3)&  (1^3)&   q\\
(1^3)&  (1^3)&  (2, 1)&   1\\
\hline
(1^4)&  (1^4)&  (1^4)&   q^3 + q\\
(1^4)&  (1^4)&  (21^2)&   q^2 + q + 1\\
(1^4)&  (1^4)&  (2^2)&   q\\
(1^4)&  (1^4)&  (3, 1)&   1\\
(1^4)&  (21^2)&  (21^2)&   q + 1\\
(1^4)&  (21^2)&  (2^2)&   1\\
(21^2)&  (21^2)&  (21^2)&   1\\
\hline
(1^5)&  (1^5)&  (1^5)&   q^6 + q^4 + q^3 + q^2 + q\\
(1^5)&  (1^5)&  (21^3)&   q^5 + q^4 + 2q^3 + 2q^2 + 2q + 1\\
(1^5)&  (1^5)&  (2^21)&   q^4 + q^3 + 2q^2 + 2q + 1\\
(1^5)&  (1^5)&  (31^2)&   q^3 + q^2 + 2q + 1\\
(1^5)&  (1^5)&  (3, 2)&   q^2 + q + 1\\
(1^5)&  (1^5)&  (4, 1)&   1\\
(1^5)&  (21^3)&  (21^3)&   q^4 + 2q^3 + 3q^2 + 4q + 2\\
(1^5)&  (21^3)&  (2^21)&   q^3 + 2q^2 + 3q + 2\\
(1^5)&  (21^3)&  (31^2)&   q^2 + q + 2\\
(1^5)&  (21^3)&  (3, 2)&   q + 1\\
(1^5)&  (2^21)&  (2^21)&   q^2 + 2q + 2\\
(1^5)&  (2^21)&  (31^2)&   q + 1\\
(1^5)&  (2^21)&  (3, 2)&   1\\
(21^3)&  (21^3)&  (21^3)&   q^3 + 3q^2 + 4q + 4\\
(21^3)&  (21^3)&  (2^21)&   q^2 + 3q + 3\\
(21^3)&  (21^3)&  (31^2)&   q + 1\\
(21^3)&  (21^3)&  (3, 2)&   1\\
(21^3)&  (2^21)&  (2^21)&   q + 2\\
(21^3)&  (2^21)&  (31^2)&   1\\
(2^21)&  (2^21)&  (2^21)&   1\\
\end{tabular}
\end{center}

\begin{center}
\begin{tabular}{LLL|L}
\mu^1&\mu^2&\mu^3& U_\muhat\\
\hline
(1) & (1) & (1) &   1\\
\hline
(1^2) & (1^2) & (1^2) &  1\\
(1^2) & (1^2) &  (2) &   1\\
(2) & (2) & (2) &   1\\
\hline
(1^3) & (1^3) & (1^3) &  q + 1\\
(1^3) & (1^3) & (2, 1) &   2\\
(1^3) & (1^3) & (3) &   1\\
(1^3) & (2, 1) & (2, 1) &   2\\
(2, 1) & (2, 1) & (2, 1) &   2\\
(2, 1) & (2, 1) & (3) &   1\\
(3) & (3) & (3) &   1\\
\hline
(1^4) & (1^4) & (1^4) &  q^3 + 2q + 1\\
(1^4) & (1^4) & (21^2) &  q^2 + 2q + 3\\
(1^4) & (1^4) & (2^2) &  q + 2\\
(1^4) & (1^4) & (3, 1) &   3\\
(1^4) & (1^4) & (4) &   1\\
(1^4) & (21^2) & (21^2) &  2q + 6\\
(1^4) & (21^2) & (2, 2) &   3\\
(1^4) & (21^2) & (3, 1) &   3\\
(1^4) & (2^2) & (2^2) &  2\\
(1^4) & (2^2) & (3, 1) &   1\\
(21^2) & (21^2) & (21^2) &  q + 9\\
(21^2) & (21^2) & (2^2) &  5\\
(21^2) & (21^2) & (3, 1) &   4\\
(21^2) & (21^2) & (4) &   1\\
(21^2) & (2^2) & (2^2) &  1\\
(21^2) & (2^2) & (3, 1) &   2\\
(21^2) & (3, 1) & (3, 1) &   2\\
(2^2) & (2^2) & (2^2) &  2\\
(2^2) & (2^2) & (3, 1) &   1\\
(2^2) & (2^2) & (4) &   1\\
(2^2) & (3, 1) & (3, 1) &   1\\
(3, 1) & (3, 1) & (3, 1) &   2\\
(3, 1) & (3, 1) & (4) &   1\\
(4) & (4) & (4) &   1\\
\end{tabular}
\end{center}

\begin{center}
\begin{tabular}{LLL|L}
\mu^1&\mu^2&\mu^3& U_\muhat\\
\hline
(1^5) & (1^5) & (1^5) &  q^6 + q^4 + 2q^3 + q^2 + 3q + 1\\
(1^5) & (1^5) & (21^3) &  q^5 + q^4 + 3q^3 + 3q^2 + 6q + 4\\
(1^5) & (1^5) & (2^21) &  q^4 + q^3 + 3q^2 + 5q + 5\\
(1^5) & (1^5) & (31^2) &  q^3 + 2q^2 + 4q + 6\\
(1^5) & (1^5) & (3, 2) &   q^2 + 2q + 5\\
(1^5) & (1^5) & (4, 1) &   4\\
(1^5) & (1^5) & (5) &   1\\
(1^5) & (21^3) & (21^3) &  q^4 + 3q^3 + 5q^2 + 11q + 12\\
(1^5) & (21^3) & (2^21) &  q^3 + 3q^2 + 8q + 12\\
(1^5) & (21^3) & (31^2) &  2q^2 + 4q + 12\\
(1^5) & (21^3) & (3, 2) &   2q + 8\\
(1^5) & (21^3) & (4, 1) &   4\\
(1^5) & (2^21) & (2^21) &  q^2 + 4q + 12\\
(1^5) & (2^21) & (31^2) &  3q + 9\\
(1^5) & (2^21) & (3, 2) &   7\\
(1^5) & (2^21) & (4, 1) &   2\\
(1^5) & (31^2) & (31^2) &  q + 6\\
(1^5) & (31^2) & (3, 2) &   3\\
(1^5) & (3, 2) & (3, 2) &   2\\
\end{tabular}
\end{center}

\begin{center}
\begin{tabular}{LLL|L}
\mu^1&\mu^2&\mu^3& U_\muhat\\
\hline
(21^3) & (21^3) & (21^3) &  2q^3 + 6q^2 + 16q + 28\\
(21^3) & (21^3) & (2^21) &  2q^2 + 10q + 26\\
(21^3) & (21^3) & (31^2) &  q^2 + 6q + 21\\
(21^3) & (21^3) & (3, 2) &   q + 15\\
(21^3) & (21^3) & (4, 1) &   6\\
(21^3) & (21^3) & (5) &   1\\
(21^3) & (2^21) & (2^21) &  4q + 22\\
(21^3) & (2^21) & (31^2) &  2q + 18\\
(21^3) & (2^21) & (3, 2) &   10\\
(21^3) & (2^21) & (4, 1) &   4\\
(21^3) & (31^2) & (31^2) &  2q + 12\\
(21^3) & (31^2) & (3, 2) &   8\\
(21^3) & (31^2) & (4, 1) &   3\\
(21^3) & (3, 2) & (3, 2) &   4\\
(21^3) & (3, 2) & (4, 1) &   1\\
(2^21) & (2^21) & (2^21) &  q + 17\\
(2^21) & (2^21) & (31^2) &  q + 13\\
(2^21) & (2^21) & (3, 2) &   8\\
(2^21) & (2^21) & (4, 1) &   4\\
(2^21) & (2^21) & (5) &   1\\
(2^21) & (31^2) & (31^2) &  11\\
(2^21) & (31^2) & (3, 2) &   6\\
(2^21) & (31^2) & (4, 1) &   2\\
(2^21) & (3, 2) & (3, 2) &   4\\
(2^21) & (3, 2) & (4, 1) &   2\\
(31^2) & (31^2) & (31^2) &  q + 10\\
(31^2) & (31^2) & (3, 2) &   7\\
(31^2) & (31^2) & (4, 1) &   4\\
(31^2) & (31^2) & (5) &   1\\
(31^2) & (3, 2) & (3, 2) &   3\\
(31^2) & (3, 2) & (4, 1) &   2\\
(31^2) & (4, 1) & (4, 1) &   2\\
(3, 2) & (3, 2) & (3, 2) &   3\\
(3, 2) & (3, 2) & (4, 1) &   2\\
(3, 2) & (3, 2) & (5) &   1\\
(3, 2) & (4, 1) & (4, 1) &   1\\
(4, 1) & (4, 1) & (4, 1) &   2\\
(4, 1) & (4, 1) & (5) &   1\\
(5) & (5) & (5) &   1\\
\end{tabular}
\end{center}

\begin{center}
\begin{tabular}{LLL|L}
\mu^1&\mu^2&\mu^3& U'_\muhat\\
\hline
(1) & (1) & (1) &  1\\
\hline
(1^2) & (1^2) & (1^2) &  1\\
(1^2) & (1^2) & (2) &  1\\
(2) & (2) & (2) &  1\\
\hline
(1^3) & (1^3) & (1^3) &  q + 1\\
(1^3) & (1^3) & (3) &  1\\
(2, 1) & (2, 1) & (3) &  1\\
(3) & (3) & (3) &  1\\
\hline
(1^4) & (1^4) & (1^4) &  q^3 + 1\\
(1^4) & (1^4) & (21^2) &  q^2 + 1\\
(1^4) & (1^4) & (2^2) &  q + 2\\
(1^4) & (1^4) & (3, 1) &  1\\
(1^4) & (1^4) & (4) &  1\\
(1^4) & (21^2) & (2^2) &  1\\
(1^4) & (21^2) & (3, 1) &  1\\
(1^4) & (2^2) & (2^2) &  2\\
(1^4) & (2^2) & (3, 1) &  1\\
(21^2) & (21^2) & (21^2) &  q + 1\\
(21^2) & (21^2) & (2^2) &  1\\
(21^2) & (21^2) & (4) &  1\\
(21^2) & (2^2) & (2^2) &  1\\
(2^2) & (2^2) & (2^2) &  2\\
(2^2) & (2^2) & (3, 1) &  1\\
(2^2) & (2^2) & (4) &  1\\
(2^2) & (3, 1) & (3, 1) &  1\\
(3, 1) & (3, 1) & (4) &  1\\
(4) & (4) & (4) &  1\\
\hline
(1^5) & (1^5) & (1^5) &  q^6 + q^4 + q^2 + q + 1\\
(1^5) & (1^5) & (21^3) &  q^5 - q^4 + q^3 - q^2\\
(1^5) & (1^5) & (2^21) &  q^4 - q^3 + q^2 + q + 1\\
(1^5) & (1^5) & (31^2) &  q^3 + 2q + 2\\
(1^5) & (1^5) & (3, 2) &  q^2 + 1\\
(1^5) & (1^5) & (5) &  1\\
(1^5) & (21^3) & (21^3) &  q^4 - q^3 + q^2 - q\\
(1^5) & (21^3) & (2^21) &  q^3 - q^2\\
(1^5) & (2^21) & (2^21) &  q^2\\
(1^5) & (2^21) & (31^2) &  q + 1\\
(1^5) & (2^21) & (3, 2) &  1\\
(1^5) & (31^2) & (31^2) &  q + 2\\
(1^5) & (31^2) & (3, 2) &  1\\
\end{tabular}
\end{center}

\begin{center}
\begin{tabular}{LLL|L}
\mu^1&\mu^2&\mu^3& U'_\muhat\\
\hline
(21^3) & (21^3) & (31^2) &  q^2 + 1\\
(21^3) & (21^3) & (3, 2) &  q + 1\\
(21^3) & (21^3) & (5) &  1\\
(21^3) & (31^2) & (4, 1) &  1\\
(21^3) & (3, 2) & (4, 1) &  1\\
(2^21) & (2^21) & (2^21) &  q + 1\\
(2^21) & (2^21) & (31^2) &  q + 1\\
(2^21) & (2^21) & (5) &  1\\
(2^21) & (31^2) & (31^2) &  1\\
(31^2) & (31^2) & (31^2) &  q + 2\\
(31^2) & (31^2) & (3, 2) &  1\\
(31^2) & (31^2) & (5) &  1\\
(31^2) & (3, 2) & (3, 2) &  1\\
(3, 2) & (3, 2) & (3, 2) &  1\\
(3, 2) & (3, 2) & (5) &  1\\
(3, 2) & (4, 1) & (4, 1) &  1\\
(4, 1) & (4, 1) & (5) &  1\\
(5) & (5) & (5) &  1\\
\end{tabular}
\end{center}

\end{document}